\documentclass{amsproc}

\usepackage{amsmath}
\usepackage{graphicx}
\graphicspath{{./figs/}} 
\usepackage{lipsum}
\usepackage{amsfonts}
\usepackage{graphicx}
\usepackage{algorithm}
\usepackage{algpseudocode}
\usepackage{hyperref}

\theoremstyle{thmstyletwo}%
\newtheorem{theorem}{Theorem}
\newtheorem{proposition}[theorem]{Proposition}%
\newtheorem{corollary}[theorem]{Corollary}

\newtheorem{assumption}{Assumption}
\newtheorem{lemma}{Lemma}

\numberwithin{equation}{section}

\begin{document}


\title[Algorithmic framework for interface related optimization]{A modularized algorithmic framework for interface related optimization problems using characteristic functions}

\author{Dong Wang}
\address{School of Science and Engineering, The Chinese University of Hong Kong, Shenzhen \& Shenzhen International Center for Industrial and Applied Mathematics, Shenzhen Research Institute of Big Data, Guangdong, China}
\curraddr{}
\email{wangdong@cuhk.edu.cn}
\thanks{}

\author{Shangzhi Zeng}
\address{Department of Mathematics and Statistics, University of Victoria, Victoria, BC, Canada}
\curraddr{}
\email{zengshangzhi@uvic.ca}
\thanks{}

\author{Jin Zhang}
\address{Department of Mathematics, SUSTech International Center for Mathematics, Southern University of Science and Technology, National Center for Applied Mathematics Shenzhen, Guangdong, China}
\curraddr{}
\email{zhangj9@sustech.edu.cn}
\thanks{Corresponding author: Jin Zhang, Department of Mathematics, SUSTech International Center for Mathematics, Southern University of Science and Technology, National Center for Applied Mathematics Shenzhen, Guangdong, China ({zhangj9@sustech.edu.cn}).}

\begin{abstract}

{In this paper, we consider the algorithms and convergence for a general optimization problem, which has a wide range of applications in image segmentation, topology optimization, flow network formulation, and surface reconstruction. In particular, the problem focuses on interface related optimization problems where the interface is implicitly described by characteristic functions of the corresponding domains. Under such representation and discretization, the problem is then formulated into a discretized optimization problem where the objective function is concave with respect to characteristic functions and convex with respect to state variables. We show that under such structure, the iterative scheme based on alternative minimization can converge to a local minimizer. Extensive numerical examples are performed to support the theory.}
\end{abstract}


\maketitle

\section{Introduction}
     
Interface related optimization problem is a fundamental problem in many applications, including problems in material science \cite{Chen_2002}, image processing \cite{Chan_2001}, topology optimization \cite{Sigmund_2013}, surface reconstruction \cite{Hong_Kai_Zhao} and so on. A lot of numerical approaches have been developed to solve such problems, including front track based methods \cite{glimm2003conservative,Unverdi_1992}, phase-field based methods \cite{Chen_2002,du2019maximum,qiao2021two,shen2018convergence}, level set based methods \cite{Osher_1988}, parametric finite element based methods \cite{li2020convergence}, two-stage thresholding \cite{cai2013two,chan2014two}, Centroidal Voronoi Tessellations (CVT) based methods \cite{du2006convergence,du1999centroidal,liu2011fast}, primal dual methods \cite{bae2011global,boykov2004experimental,yuan2010continuous,wan2012reconstructing}, and many others. Usually, solving such problems includes three ingredients: 1. representation of the interface (implicit or explicit), 2. approximation of the objective functional under the representation, and 3. approaches to minimize the approximate objective functional. In particular, the representation of the interface is the most fundamental part of a model or a method for interface related optimization problems. 

This paper focuses on a wide class of approximate interface related optimization problems where the interface is implicitly represented by indicator functions of corresponding domains. It is motivated by the MBO method for approximating mean curvature flow using indicator functions \cite{barles1995simple,merriman1992diffusion}.  Esedoglu and Otto \cite{esedoglu2015threshold}  then develop a novel interpretation using minimizing movement and generalize this type of method to multiphase flow with arbitrary surface tensions. The method has subsequently been extended to deal with many other applications, including 
image processing \cite{esedog2006threshold,Ma_2021,merkurjev2013mbo,wang2016efficient}, 
problems of anisotropic interface motions \cite{Elsey_2017,esedoglu2017convolution,merriman2000convolution}, 
the wetting problem on solid surfaces \cite{lu2021efficient,WWX2018,xu2016efficient,ying2021adaptive}, convex object segmentation \cite{luo2023binary,luo2020convex}, and so on. The method can also be considered as piecewise constant level set methods \cite{jiang2022topology,tai2007image,tai2023potts,wei2009piecewise,zhang2018approach}. 

Recently, based on Esedoglu and Otto's novel interpretation, in \cite{wang2016efficient,wang2019iterative}, the authors develop an efficient iterative convolution thresholding method (ICTM) for image segmentation and extend into topology optimization problems \cite{chen2018efficient,hu2022} and surface reconstructions from point clouds \cite{Wang_2021}. In general, the problem could be formulated into 
\begin{align}
\min_{\Theta_i \in \mathcal S, \Omega_i} \mathcal{E} \colon = \sum_{i=1}^n \int_{\Omega_i} F_i(\Theta_1, \ldots, \Theta_n) \ dx +  \sum_{i=1}^n \lambda_i |\partial \Omega_i| \label{EnergyGeneral}
\end{align}
Here, $F_i$ are usually fidelity terms, $\Theta_i = (\Theta_{i,1}, \Theta_{i,2}, \ldots, \Theta_{i,m})$ contains all possible parameters in fidelity terms,  $\Omega = \cup_{i=1}^n \Omega_i$,  $\mathcal{S} =\mathcal{S}_1 \times\mathcal{S}_2\times \ldots\times \mathcal{S}_n \cap  \mathcal{S}_1^{o} \times\mathcal{S}_2^{o} \times \ldots\times \mathcal{S}_n^{o}  $ as the admissible set of $\Theta = (\Theta_1,\Theta_2, \ldots, \Theta_n)$, $\mathcal{S}_i \cap \mathcal{S}_i^{o} $ as the admissible set of $\Theta_i$ where $\mathcal{S}_i^{o}$ is the admissible set for  satisfying some constraints of $\Theta_i$ which are dependent on  the partition of $\Omega = \cup_{i=1}^n \Omega_i$ and $\mathcal{S}_i$ is the admissible set for satisfying some constraints of $\Theta_i$ which are independent on  the partition, and $\lambda_i$ are fixed parameters. 

Denote $u_i$ to be indicator functions of $\Omega_i$ ($i \in [n]$), as pointed out in  \cite{esedoglu2015threshold}, when $\tau \ll 1$, the measure of $\partial \Omega_i$ can be approximated by
\begin{align}
|\partial \Omega_i|\approx \sqrt{\frac{\pi}{\tau}} \sum_{j\in[n], j \neq i}\int_{\Omega} u_i G_{\tau} *u_j \ dx, \label{approxPeri}
\end{align}
where $*$ represents convolution and $G_\tau$ is
\begin{align}\label{kernel}
G_{\tau}(x)=\frac{1}{4\pi\tau} \exp(-\frac{|x|^2}{4\tau}).
\end{align}

Then, problem~\eqref{EnergyGeneral} can be approximately written as
\begin{align}
\min_{\Theta_i \in \mathcal S, u_i \in \mathcal{B} } \mathcal{E}^\tau \colon = \sum_{i=1}^n \int_{\Omega} u_i F_i(\Theta_1, \ldots, \Theta_n) \ dx + \sqrt{\frac{\pi}{\tau}}  \sum_{i=1}^n \lambda_i  \sum_{j\in[n], j \neq i}\int_{\Omega} u_i G_{\tau} *u_j \ dx\label{EnergyApprox}
\end{align}
where \[\mathcal{B} \colon = \{u\in BV(\Omega,\mathbb{R}) \ |  \ u =\{0,1\} \}\]
and $BV(\Omega,\mathbb{R})$ denotes the bounded variation functional space.

We note that special cases of problem~\eqref{EnergyApprox} include but not limited to the approximate models for Chan-Vese model for image segmentation \cite{Chan_2001}, local binary fitting model (LBF) for intensity inhomogeneous image segmentation \cite{Li_2007}, local statistical active contour model \cite{Zhang_2016}, topology optimization for Stokes flow \cite{Borrvall_2002}, formulation of biological flow networks \cite{Hu_2013,hu2022}, composite materials \cite{Cherkaev_2011}, Dirichlet partition problems \cite{wang2022efficient}, and surface reconstruction from point clouds \cite{Wang_2021}.
In all these cases, $F_i$ is usually taken to be convex in the admissible set $\mathcal S$ with respect to $\Theta$.

It is not difficult to check that the second term in \eqref{EnergyApprox} is strictly concave by the following direct calculation using the fact that $G_\tau*$ is a self adjoint operator :
\begin{align*}
\sum_{j\in[n], j \neq i}\int_{\Omega} u_i G_{\tau} *u_j \ dx  =  \int_{\Omega} u_i G_{\tau} *(1-u_i) \ dx =  \int_{\Omega} (G_{\tau/2} *u_i) G_{\tau/2} *(1-u_i) \ dx.
\end{align*}

Obeying the fact that either images are defined on discretized pixels or computational domains are discretized for fluid problems, we thus arrive at the following optimization problem in a compact form

\begin{equation}\label{prb2}
	\min_{u \in {C}, \theta \in \mathcal S } \Phi(u,\theta) := h(\theta) + \sum_i \sum_j u_{i,j}{f_{i,j}(\theta)} + g(u),
\end{equation}
where $h$ is a convex (can be nonsmooth) function, $f_{i,j}$ are convex functions with respect to $\theta$ for each $i$ and $j$, $h$ and $f_{i,j}$ are continuous on $\mathcal S$, $\mathcal S$ is the admissible set for $\theta$, $g$ is a strictly concave function and $C$ is the admissible set for $u:= (u_1;u_2; \ldots; u_n)\in \mathbb{R}^{n \times p}$;
\begin{align}
C := \{ u ~|~ u_i = (u_{i,1},u_{i,2},\cdots, u_{i,p})  \in \{0,1\}^p, \,\, \sum_{i=1}^n u_{i,j} = 1,  \ \  \forall \ j \in [p]\}.
\end{align}
In particular, $n$ (usually $n\geq 2$) denotes the number of phases in different applications (for example, number segments in image segmentation or number of materials in optimal composite of materials) and $p$ usually describes the degree of freedom for problems after discretization. In addition, summation equals $1$ constraint implies that at each fixed point (say point $j$) of discretization, only one index $i$ such that $u_{i,j}$ takes $1$ and all others take $0$.

Existing ICTM  for \eqref{prb2} usually achieves stable solutions in a couple of steps, which is empirically observed in the literature; see  \cite{chen2018efficient,wang2016efficient,wang2019iterative}. However, theoretical guarantee of convergence towards solutions with ``good" qualities is still lacking. In response to this limitation, 
in this paper, we develop a convex relaxation based alternating minimization 
 algorithmic framework for solving \eqref{prb2}. 
We justify the validity of our scheme by that the relaxation yields exact solutions.
This new framework is
flexible and modularized to handle a
broad class of application-driven problems in the form of \eqref{prb2}.
We then ask, is it possible to achieve stationary or even local solutions? We
provide a positive answer to this question by separating our discussion for different embedded modules.

In this paper, we give a theoretical and systematical study on the convergence of the ICTM. 
The convergence to a local minimizer is rigorously proved  in accordance with three practical situations. Among existing literature of interface related optimization problems, only energy decaying property and the convergence of the value of objective function have been studied. However, there is no study on the convergence of the generated sequence itself.  As for related interface motion problems where only the perimeter term is minimized, existing work mainly considers the convergence on the dynamics when the approximation parameter ({\it i.e.}, $\tau$ in \eqref{approxPeri}) goes to $0$; see \cite{laux2016convergence2}.  Indeed, individual treatments on the three situations enable us to take advantage of their special
algorithmic structures more effectively and thus to derive some new results.  This is
a striking feature of our study that leads to 
 a big difference between our convergence analysis and others. To our knowledge, this is the first and systematical convergence result of the ICTM for interface related optimization problems.

The paper is organized as follows. The detail on the convergence is introduced in Section~\ref{sec:conv} and numerical experiments for different applications are performed in Section~\ref{sec:exp}. We draw some conclusion and discussions in Section~\ref{sec:con}.

\section{Algorithmic framework}\label{sec:alg}

The nonconvex optimization problem \eqref{prb2}  is in general 
computationally challenging. The difficulty is mainly due to the binary variables. 
To develop efficient solution scheme for solving problem \eqref{prb2}, we first include a convex relaxation of the binary constraint:
\begin{equation}\label{prb3}
	\min_{u \in {\hat{C}}, \theta \in \mathcal S } \Phi(u,\theta) := h(\theta) + \sum_i \sum_j u_{i,j}{f_{i,j}(\theta)} + g(u),
\end{equation}
where $\mathcal S$ is a closed convex set and $\hat C$ is the convex hull of the binary constraint $C$,
\[\hat C := \{ u ~|~ u_i = (u_{i,1},u_{i,2},\cdots, u_{i,n})  \in [0,1]^p,\ \  \sum_{i=1}^n u_{i,j} = 1,  \forall \ j \in [p]\}.\] 

Before we consider continuous optimization algorithms for solving problem \eqref{prb3},  it is hoped that the convex relaxation yields exact solution. Thanks to the strict concavity of $g$,  the following results exhibit the relaxation exactness
by showing that any relaxed local solution maps to
a true solution.

\begin{lemma}\label{lem_r}
	For any $\bar{\theta} \in \mathcal S$, let $\bar{u}$ be a local minimum of problem $$\underset{u \in \hat{C}}{\mathrm{argmin}} ~ \Phi (u,\bar{\theta}),$$ that is, there exists $\epsilon > 0$ such that
	\[
	\Phi (\bar{u},\bar{\theta}) \le \Phi (u,\bar{\theta}), \quad \forall u \in \hat{C} \cap \mathbb{B}_{\epsilon}(\bar{u}),
	\]
	where $\mathbb{B}_{\epsilon}(\bar{u}) := \{ u ~|~ \|u - \bar{u} \| \le \epsilon \}$.
	Then $\bar{u} \in C$.
\end{lemma}
\begin{proof} We prove the result for by contradiction.
	Assume that there exist $i$ and $j$ such that $\bar{u}_{i,j} \notin \{0,1\}$, then the set $$A := \{j ~|~ \bar{u}_{i,j} \notin \{0,1\}~\text{for some}~i \}$$ is nonempty.
	Since $\sum_i \bar{u}_{i,j} = 1$, there exist $\bar{j} \in A$ and $i_1, i_2 \in \{1,\ldots,n\}$ such that $\bar{u}_{i_1,\bar{j}}, \bar{u}_{i_2,\bar{j}} \in (0,1)$. And we can find $c_0 \in (0,\frac{1}{2})$ such that $$c_0 \le \bar{u}_{i_1,\bar{j}}, \ \ \bar{u}_{i_2,\bar{j}}  \le 1-c_0.$$
	
	Define $u(t) = (u_1(t); u_2(t); \ldots; u_n(t))$, where 
	$$u_{i_1,\bar{j}}(t) = \bar{u}_{i_1,\bar{j}} + t,  \ \ u_{i_2,\bar{j}}(t) = \bar{u}_{i_2,\bar{j}} - t, \ \  u_{i,j}(t) = \bar{u}_{i,j}, \ \  \textrm{if}  \ \ i \notin \{i_1, i_2\}, \ j\neq \bar{j}. $$ 
	When $-c_0 \le t \le c_0$, we have that $0 \le u_{i,j}(t) \le 1$ for all $i$ and $j$, and $ \sum_i u_{i,j}(t) = 1$ for all $j$, which yields that $u(t) \in \hat{C}$. 
	
	Define further $\varphi(t):= \Phi(u(t),\bar{\theta})$, then $\varphi(t)$ is strictly concave since $\Phi$ is strictly concave with respect to $u$. 
	Since $\bar{u}$ is a local minimum to $\min_{u \in \hat{C}} ~ \Phi (u,\bar{\theta})$, immediately $0$ is a local minimum to problem $$\min_{t \in [-c_0,c_0]} \varphi(t),$$
which contradicts to the strict concavity of $\varphi$.
\end{proof}

The relaxation exactness regarding local minimums follows immediately.
\begin{proposition}
	Let $(\bar{u}, \bar{\theta})$ be a local minimum of problem \eqref{prb3},
	then $\bar{u}_{i,j} \in \{0,1\}$ for any $i$ and $j$.
\end{proposition}

We are now in the position to solve the relaxed problem \eqref{prb3}.
We present an alternating minimization algorithmic framework. That is, the algorithms are based on generating the following updating sequence starting with an initial guess $u^0$;
\[\theta^0, u^1, \theta^1, u^2, \theta^2, \cdots, u^n, \theta^n, \cdots\]
where 
\begin{align}
\theta^k = \arg\min_{\theta\in \mathcal S} \Phi(u^k,\theta) \label{subproblem1}
\end{align}
and
\begin{align}
u^{k+1} = \arg\min_{u\in \hat C} \Phi(u,\theta^k). \label{subproblem2}
\end{align}

\section{Convergence results}\label{sec:conv}

The algorithmic framewrok stated in the preceding section consists of sequentially solving subproblems (\ref{subproblem1}) and (\ref{subproblem2})  with respect to one variable when the other variable is fixed in an alternative manner. 
Naturally, for different structures of $h, f_{i,j}$, and $g$, we may utilize efficiently their structural information and hence call suitable solution schemes to tackle subproblems (\ref{subproblem1}) and (\ref{subproblem2}).  Thereby, our convergence analysis in this section mainly relates to three cases where we employ appropriate method for minimizations which aims at finding for each $k$, 
\begin{itemize}
\item [1.] a global minimizer of convex subproblem \eqref{subproblem1} and a global minimizer of subproblem \eqref{subproblem2},
\item [2.] a global minimizer of convex subproblem \eqref{subproblem1} and a local minimizer of subproblem \eqref{subproblem2},
\item [3.] an approximate stationary solution of convex subproblem \eqref{subproblem1} and a thresholding approach for subproblem  \eqref{subproblem2}.
\end{itemize}

The main convergence results for the three situations, namely {\bf Global $\theta$ $\&$ Global $u$ case, Global $\theta$ $\&$ Local $u$ case and inexact Stationay  $\theta$ $\&$ Thresholding $u$ case} are presented in Theorems \ref{thm_min}, \ref{thm_min2}, \ref{thm_conv_inexact} and \ref{thm_min3}.

\subsection{Convergence analysis for Global $\theta$ $\&$ Global $u$ case}
We first analyze the alternating minimization method on problem \eqref{prb2} in the situation where both global minimizers of subproblems \eqref{subproblem1} and \eqref{subproblem2} are accessible. The algorithm update is summarized in Algorithm~\ref{alg1}.

\begin{algorithm}[ht!]
\begin{flushleft}

 {\bf Input:} ($u^0, \theta^0$)

{\bf Output:} ($\bar u, \bar \theta$)  that approximately minimizes  \eqref{prb2} .
	
	Set $k=0$\;
	
{\bf While} the following condition fail to hold: \begin{equation}\label{stop1}
u^k \in \mathrm{argmin}_{u \in { \hat{C}}} \Phi(u,\theta^k),
\end{equation}

{
		{\bf 1.} Find $u^{k+1}$ such that
		 \begin{equation}\label{AM1}
		u^{k+1} \in \mathrm{argmin}_{u \in { \hat{C}}} \Phi(u,\theta^k) =  \sum_i \sum_j u_{i,j}{f_{i,j}(\theta^k)} + g(u).
		\end{equation}		
		{\bf 2.} Find $\theta^{k+1}$ such that
		\begin{equation}\label{AM2}
		\theta^{k+1} = \mathrm{argmin}_{\theta \in \mathcal S} \Phi(u^{k+1},\theta) = h(\theta) + \sum_i \sum_j u_{i,j}^{k+1}{f_{i,j}(\theta)}.
		\end{equation}

		 Set $k = k+1$\;
	}
\end{flushleft}

	\caption{An iterative scheme with Global $\theta$ $\&$ Global $u$ for \eqref{prb2} }  \label{alg1}
\end{algorithm}

Note that, as shown in Lemma \ref{lem_r}, that $\mathrm{argmin}_{u \in { \hat{C}}} \Phi(u,\theta^k) \subset C $. Hence, condition \eqref{stop1} is equivalent to that $u^{k}  \in \mathrm{argmin}_{u \in C} \Phi(u, \theta^k)$. Furthermore, we can derive the following result immediately.

\begin{lemma}
	Let $\{(u^k, \theta^k)\}$ be the sequence generated by \eqref{AM1} and \eqref{AM2}, then ${u}_{i,j}^k \in \{0,1\}$ for any $i$, $j$ and $k$.
\end{lemma}

The subsequent result establishes the stability of the solution set of $u$-minimization problem with respect to variable $\theta$, denoted as $C^*(\theta) := \mathrm{argmin} _{u \in \hat{C}} \Phi(u, {\theta}) $.

\begin{lemma}\label{lem_min}
	For any $\bar{\theta} \in \mathcal S$ and $\bar{u} \in \mathrm{argmin}_{u \in\hat{C}} \Phi(u, \bar{\theta})$, then there exists $\epsilon > 0$ such that for any $\theta_\epsilon \in \mathbb{B}_{\epsilon}(\bar{\theta}) \cap \mathcal S$,  {where $\mathbb{B}_{\epsilon}(\bar{\theta}) := \{ \theta ~|~ \|\theta - \bar{\theta} \| \le \epsilon \}$,} $\mathrm{argmin}_{u \in\hat{C}} \Phi(u,\theta_\epsilon) \subseteq C^*(\bar{\theta}) := \mathrm{argmin} _{u \in \hat{C}} \Phi(u, \bar{\theta}) $.
\end{lemma}
\begin{proof}
	From the definition of $C$ we have $|C| < \infty$.
	As $\bar{u} \in C^*(\bar{\theta}) $, 
	\[\Phi(\bar{u}, \bar{\theta}) < \Phi(u, \bar{\theta}), \  \forall u \in C\backslash C^*(\bar{\theta}).\] 
	
	Furthermore, because $|C| < \infty$, there exists $\delta > 0$ such that $\Phi(\bar{u}, \bar{\theta}) < \Phi(u, \bar{\theta}) - \delta$ for any $u \in C\backslash C^*(\bar{\theta})$. 
	
	Since $h$ and $f_{i,j}$ are continuous on $\mathcal S$, for each $u \in C\backslash C^*(\bar{\theta})$, there exists $\epsilon_u > 0$ such that $\Phi(\bar{u}, \theta_\epsilon) < \Phi(u, \theta_\epsilon) - \delta/2$ for any $\theta_\epsilon \in \mathbb{B}_{\epsilon_u}(\bar{\theta}) \cap \mathcal S$. In addition, because $|C| < \infty$, there exists $\epsilon > 0$ such that $\min_{u \in C\backslash C^*(\bar{\theta})} \epsilon_u \ge \epsilon$ and thus \[ \Phi(\bar{u}, \theta_\epsilon) < \Phi(u, \theta_\epsilon) - \delta/2, \ \ \ \forall u \in C\backslash C^*(\bar{\theta})\] and $\theta_\epsilon \in \mathbb{B}_{\epsilon}(\bar{\theta}) \cap \mathcal S$.
	
	Combining above with Lemma \ref{lem_r} yields that 
	\[ \mathrm{argmin}_{u \in\hat{C}} \Phi(u,\theta_\epsilon) \subseteq C^*(\bar{\theta}),   \ \ \ \forall \theta_\epsilon \in \mathcal S \cap \mathbb{B}_{\epsilon}(\bar{\theta}).\] 
\end{proof}

{
	To analyze the convergence property of the sequence generated by alternating minimization scheme  \eqref{AM1} and \eqref{AM2}, we introduce the following assumption on the problem \eqref{prb2}. This assumption is considered mild, particularly for the problem \eqref{prb2} derived from the problem~\eqref{EnergyGeneral}, where the functions $f_{i,j}(\theta)$ typically vary distinctively for each $i, j$.
	
	\begin{assumption}\label{ass1}
		For any $u_1, u_2 \in C$ and $\theta \in \mathcal S$, if $u_1 \neq u_2$, then $	 \Phi(u_1,\theta) \neq  \Phi(u_2,\theta)$.
	\end{assumption}

It can be established that, under Assumption \ref{ass1}, $C^*(\theta) $ is a singleton for any $\theta \in \mathcal S$.

\begin{proposition}\label{prop_min}
	Suppose Assumption \ref{ass1} is satisfied. Let $\{(u^k, \theta^k)\}$ be the sequence generated by alternating minimization scheme  \eqref{AM1} and \eqref{AM2}, suppose that condition \eqref{stop1} is satisfied for some $k$, then $(u^{k}, \theta^{k})$ is a local minimizer of problem \eqref{prb2} in the sense that there exists $\epsilon > 0$ such that
	\[
	\Phi(u^{k}, \theta^{k}) \le \Phi(u, \theta), \quad \forall (u,\theta) \in \hat{C} \times (\mathcal S \cap \mathbb{B}_{\epsilon}(\theta^{k})).
	\]
\end{proposition}
\begin{proof}
	Let $(\bar{u}, \bar{\theta})$ denote the iterate $(u^k, \theta^k)$ satisfying conditions \eqref{stop1}.
	Given the fulfillment of Assumption \ref{ass1}, $C^*(\bar{\theta})$ is a singleton, yielding $C^*(\bar{\theta}) = \{\bar{u}\}$.
	Lemma \ref{lem_min} shows that there exists $\epsilon > 0$ such that 
	$$\mathrm{argmin}_{u \in\hat{C}} \Phi(u,\theta_\epsilon) \subseteq C^*(\bar{\theta})= \{\bar{u}\}, \ \ \forall \theta_\epsilon \in \mathcal S \cap \mathbb{B}_{\epsilon}(\bar{\theta}).$$
	If there exists $(u_0, \theta_0) \in \hat{C} \times (\mathcal S \cap \mathbb{B}_{\epsilon}(\bar{\theta}))$ such that $$\Phi(u_0, \theta_0) < \Phi(\bar{u}, \bar{\theta}),$$ then $\Phi(\bar{u}, \theta_0) \le \Phi(u_0, \theta_0)$ because of the fact that 
	$$\mathrm{argmin}_{u \in \hat{C}} \Phi(u,\theta_0) \subseteq C^*(\bar{\theta}) = \{\bar{u}\}.$$
	In addition, because $\bar{\theta} = \mathrm{argmin}_{\theta \in \mathcal S} \Phi(\bar{u},\theta)$, we have \[\Phi(\bar{u}, \bar{\theta})  \le \Phi(\bar{u}, \theta_0) \le \Phi(u_0, \theta_0) < \Phi(\bar{u}, \bar{\theta}),\] which leads a contradiction and we get the conclusion.
\end{proof}

Next, we have following result for the fixed point of Algorithm~\ref{alg1}.

\begin{theorem}\label{thm_min}
	Suppose Assumption \ref{ass1} is satisfied. Let $\{(u^k, \theta^k)\}$ be the sequence generated by alternating minimization scheme \eqref{AM1} and \eqref{AM2}, then ${u}_{i,j}^k \in \{0,1\}$ for any $i$, $j$ and $k$, and
	\[
	\Phi(u^{k+1}, \theta^{k+1}) < \Phi(u^{k}, \theta^{k}),
	\]
	provided condition \eqref{stop1} is not satisfied at iteration $k$. Furthermore, if $\Phi$ is bounded below on $C \times \mathcal{S}$, there exists $K \ge 0$ such that $(u^{K}, \theta^{K})$ satisfies condition \eqref{stop1}, and is a local minimizer of problem \eqref{prb2} in the sense that there exists $\epsilon > 0$ such that
	\[
	\Phi(u^{K}, \theta^{K}) \le \Phi(u, \theta), \quad \forall (u,\theta) \in \hat{C} \times (\mathcal S \cap \mathbb{B}_{\epsilon}(\theta^{K})).
	\]
\end{theorem}
\begin{proof}
	Given the fulfillment of Assumption \ref{ass1}, $C^*({\theta})$ is a singleton for any $\theta \in \mathcal S$.
	Becuase $u^{k+1} \in \mathrm{argmin}_{u \in \hat{C}} \Phi(u,\theta^k)$, we have that $C^*({\theta^k}) = \{u ^{k+1}\}$. When condition \eqref{stop1} is not satisfied, {\it i.e.}, $u^{k+1} \neq u^k$, it follows that
	 \[
	\mathrm{min}_{\theta \in \mathcal S} \Phi(u^{k+1},\theta) =  \Phi(u^{k+1}, \theta^{k+1}) \le \Phi(u^{k+1}, \theta^{k}) < \Phi(u^{k}, \theta^{k}) = \mathrm{min}_{\theta \in \mathcal S} \Phi(u^{k},\theta).
	\]
	 
The fact that ${u}_{i,j}^k \in \{0,1\}$ for any $i$, $j$ and $k$ follows from Lemma \ref{lem_r}.  Thus $u^k \in C$ for any $k$.
	Finally, recall the fact that {$|C|$} is finite, combining with Proposition \ref{prop_min}, the conclusion easily follows.
\end{proof}

}

\subsection{Convergence analysis for Global $\theta$ $\&$ Local $u$ case}

A weakness of Proposition \ref{prop_min} is that it relates exclusively to global (as opposed to local) minima of subproblem 
(\ref{subproblem2}). This part remedies the situation somewhat.   
We replace the $u$-update \eqref{AM1} by only reaching a local minimum on $C$ in sense of
\begin{equation}\label{wAM1}
	u^{k+1} \in C_{r}^*(\theta^k),
\end{equation}	
where 
\[
C_{r}^*(\theta) := \{ \bar{u} \in C ~|~ \Phi(\bar{u},\theta) \le \Phi(u,\theta), \ \forall u \in C \cap \mathbb{B}_{r}(\bar{u}) \},
\]
with given $r > 0$.
We then accordingly replace condition \eqref{stop1} by 
\begin{equation}\label{wstop1}
u^k \in C_{r}^*(\theta^k).
\end{equation}
The algorithm is thus illustrated in Algorithm \ref{alg2}. 

\begin{algorithm}[ht!]
	\begin{flushleft}
	{\bf Input} ($u^0, \theta^0$) and $r > 0$.
	
	{\bf Output} ($\bar u, \bar \theta$)  that approximately minimizes  \eqref{prb2} .
	
	Set $k=0$\;
	
	{\bf While}{ the following condition doesn't hold: \begin{equation}\label{lstop1}
			u^k \in C_{r}^*(\theta^k),
		\end{equation}
	}{
		{\bf 1.} Find $u^{k+1}$ being a local minimizer in sense of
		\begin{equation}\label{lAM1}
			u^{k+1} \in C_{r}^*(\theta^k),
		\end{equation}		
		{\bf 2.} Find $\theta^{k+1}$ such that
		\begin{equation}\label{lAM2}
			\theta^{k+1} = \mathrm{argmin}_{\theta \in \mathcal S} \Phi(u^{k+1},\theta) = h(\theta) + \sum_i \sum_j u_{i,j}^{k+1}{f_{i,j}(\theta)}.
		\end{equation}
		Set $k = k+1$\;
	}
\end{flushleft}
	\caption{An iterative scheme with Global $\theta$ $\&$ Local $u$  for \eqref{prb2} }  \label{alg2}
\end{algorithm}

Similar to  Proposition \ref{prop_min} and Theorem \ref{thm_min}, we may derive convergence results as follows. The proofs are purely technical and thus omitted.

\begin{proposition}\label{prop_wmin}
	Suppose Assumption \ref{ass1} is satisfied. Let $\{(u^k, \theta^k)\}$ be the sequence generated by alternating minimization scheme \eqref{lAM1} with given $r > 0$ and \eqref{lAM2}, suppose that condition \eqref{lstop1} is satisfied for some $k$, then $(u^{k}, \theta^{k})$ is a local minimizer of problem \eqref{prb2} in the sense that there exists $\epsilon > 0$ such that
	\[
	\Phi(u^{k}, \theta^{k}) \le \Phi(u, \theta), \quad \forall (u,\theta) \in \left( C \cap \mathbb{B}_{r}(u^k) \right) \times (\mathcal S \cap \mathbb{B}_{\epsilon}(\theta^{k})).
	\]
\end{proposition}

\begin{theorem}\label{thm_min2}
	Suppose Assumption \ref{ass1} is satisfied. Let $\{(u^k, \theta^k)\}$ be the sequence generated by alternating minimization scheme \eqref{lAM1} with given $r > 0$ and \eqref{lAM2}, then 
	\[
	\Phi(u^{k+1}, \theta^{k+1}) < \Phi(u^{k}, \theta^{k}),
	\]
	provided conditions \eqref{lstop1} is not satisfied at iteration $k$. Furthermore, if $\Phi$ is bounded below, there exists $K \ge 0$ such that $(u^{K}, \theta^{K})$ satisfies condition \eqref{lstop1}  and is a local minimizer of problem \eqref{prb2} in the sense that there exists $\epsilon > 0$ such that
	\[
	\Phi(u^{K}, \theta^{K}) \le \Phi(u, \theta), \quad \forall (u,\theta) \in \left({C} \cap \mathbb{B}_r(u^{K}) \right)\times (\mathcal S \cap \mathbb{B}_{\epsilon}(\theta^{K})).
	\]
\end{theorem}

\subsection{Convergence analysis for  inexact Stationay $\theta$ $\&$ Thresholding $u$ case}

Both Algorithms \ref{alg1} and \ref{alg2} assume that a global or local minimum $u^k$ of subproblem \eqref{AM1} or \eqref{lAM1}, and  a global minimum $\theta^k$ of subproblem (\ref{subproblem1}) can be found exactly.
In this part, assuming that $\Phi$ is continuously differentiable with respect to $u$,  and denoting $\nabla_{u_{i,j}} \Phi$ as the gradient of $\Phi$ with respect to $u_{i, j}$, we update the variable $u$ by applying the very simple thresholding approach to solve the $u$-subproblem $$\mathrm{min}_{u \in { {C}}} \Phi(u,\theta^k).$$
To elaborate, given $u^{k,t} \in C$, for any $j \in [p]$, we choose $I(j) \in \mathrm{argmin}_{i \in [n]}  \nabla_{u_{i,j}} \Phi(u^{k,t},\theta^k)$ and set
\begin{equation}\label{th_u}
	u_{i,j}^{k, t+1} = 
	\begin{cases}
		1 & \textrm{for} \ \ i = I(j) , \\
		0 & \textrm{otherwise.}
	\end{cases}
\end{equation}
This $u^{k, t+1}$ is also a solution to the following approximation problem of $\mathrm{min}_{u \in { {C}}} \Phi(u,\theta^k)$, where the objective  is the linearization of $\Phi(u,\theta^k)$ at $u^{k, t}$,
\[
u^{k, t+1}  \in \mathrm{argmin}_{u \in C} \Phi(u^{k, t},\theta^k) + \langle \nabla_u\Phi(u^{k, t},\theta^k), u - u^{k, t} \rangle. 
\] 
On the other hand, unconstrained convex minimization method is usually solved iteratively, and only approximate first-order stationary solutions can be pursued by internal iterations.  As a consequence, it is mathematically important to involve a suitable inexactness criterion for solving convex subproblem (\ref{subproblem1})  and study the convergence with $\theta$ being accessed inexactly.
To be specific,  in $\theta$-update, we compute an approximate stationary point to subproblem (\ref{subproblem1}) 
\[\min_{\theta \in \mathcal S} \Phi(u^{k+1},\theta) = h(\theta) + \sum_i \sum_j u_{i,j}^{k+1}{f_{i,j}(\theta)}\] satisfying
\begin{equation}\label{iAM2}
\Phi(u^{k+1},\theta^{k+1}) \le \Phi(u^{k+1},\theta^k), \quad \text{and} \quad	\mathrm{dist}\left(0, \partial_\theta \Phi(u^{k+1},\theta^{k+1}) + \mathcal{N}_{\mathcal S}(\theta^{k+1}) \right) \le \varepsilon_{k+1}, 
\end{equation}
where in the sense of convex analysis, $\partial_\theta \Phi$ denotes the subgradient of $\Phi$ with respect to $\theta$, $\mathcal{N}_{\mathcal S}(\theta)$ denotes the normal cone of $\mathcal S$ at $\theta$, {i.e., $\mathcal{N}_{\mathcal S}(\theta) := \{ \xi ~|~ \langle \xi, s - \theta \rangle \le 0, \forall s \in \mathcal{S} \}$,} and $\{\varepsilon_k\}$ is a positive vanishing sequence which controls the accuracy. 
The algorithm is illustrated in Algorithm \ref{alg3}. 

\begin{algorithm}[ht!]
	\begin{flushleft}
	{\bf Input:} {$(u^0, \theta^0) \in C \times \mathcal{S}$, and sequence $\{\varepsilon_k\}$ that $\varepsilon_k \rightarrow 0$}
	
	{\bf Output:} { ($\bar u, \bar \theta$)  that approximately minimizes  \eqref{prb2} .}
	
	Set $k=0$\;
	
	{\bf While} { $\|u^{k+1} - u^k\| + \|\theta^{k+1} - \theta^k\| > tol$
	}{
		
		{\bf 1.} Find $u^{k+1} \in C$  by using \eqref{th_u} with $ u^{k,0} =  u^{k}$ until $u^{k,T+1}  = u^{k,T}$. Set $u^{k+1} = u^{k,T}$.
		
		{\bf 2.} Find $\theta^{k+1}$ being an approximate stationary point to problem $$\mathrm{min}_{\theta \in \mathcal S} \Phi(u^{k+1},\theta)$$ satisfying \eqref{iAM2}.
		
		Set $k = k+1$\;
	}
\end{flushleft}
	\caption{An iterative scheme with inexact Stationay $\theta$ $\&$  Local $u$ for \eqref{prb2} }  \label{alg3}
\end{algorithm}

\subsubsection{Analysis on the thresholding on $u$ }
The following lemma presents the basic property of the thresholding on $u $ given in \eqref{th_u}.

\begin{lemma}\label{lem1}
	Given $u \in C$ and $\theta \in \mathcal{S}$, let
	\[
	u^+ \in  \mathrm{argmin}_{v \in C} \Phi(u,\theta) + \langle \nabla_u\Phi(u,\theta), v - u \rangle,
	\]
	which, for any $j \in [p]$, can be expressed as
	\begin{equation*}
		u_{i,j}^{+} = 
		\begin{cases}
			1 & \textrm{for} \ \ i = I(j) , \\
			0 & \textrm{otherwise},
		\end{cases}
	\end{equation*}
where $I(j) \in \mathrm{argmin}_{i \in [n]}  \nabla_{u_{i,j}} \Phi(u,\theta)$.
Then 
$$\Phi(u^{+}, \theta) < \Phi(u, \theta),$$ if $u^{+} \neq u$. Moreover, $u$ is a strict local minimum of $\mathrm{min}_{u \in { \hat{C}}} \Phi(u,\theta)$ if $ u^{+} = u$ and $\mathrm{argmin}_{i \in [n]} \nabla_{u_{i,j}}\Phi(u, \theta) $ is a singleton for any $j \in [p]$.
\end{lemma}
\begin{proof}
	Given $u \in C$, if $u^+ \neq u$, since $\Phi(u, \theta)$ is strictly concave with respect to the variable $u$, it follows from \cite[Exercise 16(a), Section 3.1]{borwein2006convex} that 
	\[
	\Phi(u^+,\theta) < \Phi(u,\theta) + \langle \nabla_u\Phi(u,\theta), u^+ - u \rangle \le \Phi(u,\theta),
	\]
	where the last inequality follows from $ u^+ \in  \mathrm{argmin}_{v \in C} \Phi(u,\theta) + \langle \nabla_u\Phi(u,\theta), v - u \rangle$.
	
	Next, for any $u \in C$, if $u^{+} = u$, we will have $$I_u(u,j):= \{ i \in [n] ~|~ u_{i,j} = 1 \} \subseteq \mathrm{argmin}_{i \in [n]} \nabla_{u_{i,j}}\Phi(u, \theta) .$$
	According to the assumption that $\mathrm{argmin}_{i \in [n]} \nabla_{u_{i,j}}\Phi(u, \theta) $ is a singleton for any $j \in [p]$, we have $$ \nabla_{u_{I_u(u,j),j}} \Phi(u, \theta) < \nabla_{u_{i,j}} \Phi(u, \theta) $$ for any $j \in [p]$ and $i \neq I_u(u,j)$. 
	Then for any $d \in C - u := \{ v - u ~|~ v \in C\}$ satisfying $d \neq 0$, it holds that 
	\[
	\langle \nabla_u \Phi(u, \theta), d \rangle = \sum_{j = 1}^p \left(  \nabla_{u_{I_u(v,j),j}} \Phi(u, \theta) - \nabla_{u_{I_u(u,j),j}} \Phi(u, \theta) \right)  > 0,
	\]
	where $I_u(v, j):= \{ i \in [n] ~|~ v_{i,j} = 1 \}$.
	Therefore, there exists $\epsilon_d >0$ such that $$\Phi(u + sd, \theta) > \Phi(u, \theta), $$ for any $s \in (0,\epsilon_d)$. 
	
	Since set $C$ is with finite elements, we can find a constant $\epsilon > 0$ such that $\Phi(u + sd, \theta) > \Phi(u, \theta) $ for any $d \in  C - u $ and $s \in (0,\epsilon)$. As $\hat{C} = \mathrm{conv}(C)$, thus $\hat{C}-u = \mathrm{conv}(C - u)$ and hence there exists $\epsilon_0 > 0$ such that $$\hat{C}\cap \mathbb{B}_{\epsilon_0}(u) - u = \left( \hat{C}  - u \right)\cap \mathbb{B}_{\epsilon_0}(0) \subseteq  \mathrm{conv}(C - u) \cap \mathbb{B}_{\epsilon}(0).$$
	Then because $\Phi(u + sd, \theta) > \Phi(u, \theta) $ for any $d \in C - u $ and $s \in (0,\epsilon)$ and $\Phi$ is strictly concave with respect to $u$, we get that $\Phi(v, \theta) > \Phi(u, \theta) $ for any $v \in \hat{C}\cap \mathbb{B}_{\epsilon_0}(u) \backslash \{u\}$.  
\end{proof}

\begin{proposition}\label{localmu}
	Let $\{u^{k,t}\}_t$ be the sequence generated by \eqref{th_u} with $u^{k,0} \in C$. Then, there exists a finite $T > 0$ such that $u^{k,T+1}  = u^{k,T}$ and $\Phi(u^{k, T}, \theta^k) \le \Phi(u^{k,0}, \theta^k)$. Additionally, if $\mathrm{argmin}_{i \in [n]} \nabla_{u_{i,j}}\Phi(u^{k,T},\theta^k)$ is a singleton for any $j \in [p]$, then $u^{k,T}$ is a local minimum of $\mathrm{min}_{u \in { \hat{C}}} \Phi(u,\theta^k)$.
\end{proposition}
\begin{proof}
	As shown in Lemma \ref{lem1}, if $u^{k,t+1} \neq u^{k,t} $, we will have  $$\Phi(u^{k,t+1},\theta^k) < \Phi(u^{k,t},\theta^k).$$
	
	Then, since $u^{k,t} \in C$  and $|C|$ is finite, there must exist $T > 0$ such that $u^{k,T+1} = u^{k,T}$, and the conclusion follows from Lemma \ref{lem1}.
\end{proof}

\subsubsection{Convergence analysis}

In this part, we present the basic convergence result for Algorithm \ref{alg3}. Firstly, as established in Proposition \ref{localmu}, $\Phi(u^{k+1}, \theta^k) \le \Phi(u^k, \theta^k)$. This, combined with the fact that $\theta^{k+1}$ satisfies \eqref{iAM2}, ensures that $\Phi(u^{k+1}, \theta^{k+1}) \le \Phi(u^k, \theta^k)$, indicating a non-increasing energy trend throughout the iteration. Considering that $u^k$ is within the bounded set $C$, once $\Phi(u, \theta)$ is coercive with respect to variable $\theta$ for any $u \in C$, that is,  for any $u \in C$, $\Phi(u, \theta) \rightarrow \infty$ as $\| \theta\| \rightarrow \infty$, the boundedness of the generated sequence $\{(u^k, \theta^k)\}$ is ensured, thereby guaranteeing the existence of the limit point of the sequence $\{(u^k, \theta^k)\}$ generated by Algorithm \ref{alg3}.

\begin{theorem}\label{thm_conv_inexact}
	Let $\{(u^k, \theta^k)\}$ be the sequence generated by alternating minimization Algorithm \ref{alg3} that satisfies condition \eqref{iAM2} with  $\varepsilon_k \rightarrow 0$.
	Consider any limit point  $(\bar{u}, \bar{\theta})$ of sequence $\{(u^k, \theta^k)\}$. Then $\bar{u} \in C$.
	 Suppose that $\|\theta^{k+1}-\theta^k\| \rightarrow 0$, and that $\mathrm{argmin}_{i \in [n]} \nabla_{u_{i,j}}\Phi(\bar{u}, \bar{\theta})$ is a singleton for any $j \in [p]$,
	 then $(\bar{u}, \bar{\theta})$  is a local minimizer of problem \eqref{prb2} in the sense that there exists $\epsilon > 0$ such that
	\[
	\Phi(\bar{u}, \bar{\theta}) \le \Phi(u, \theta), \quad \forall (u,\theta) \in \left(  \hat{C} \cap \mathbb{B}_{{\epsilon}}(\bar{u})  \right) \times \left( \mathcal{S} \cap \mathbb{B}_{\epsilon}(\bar{\theta}) \right) .
	\]
\end{theorem}
\begin{proof}
	Let $(\bar{u}, \bar{\theta})$ be any limit point of sequence $\{(u^k, \theta^k)\}$ and $\{(u^l, \theta^l)\}$ be the subsequence of $\{(u^k, \theta^k)\}$ such that $(u^l, \theta^l) \rightarrow (\bar{u}, \bar{\theta})$. Since for each $k$, $\theta^{k}$ satisfies \eqref{iAM2}, there exists $\xi^{k} \in \partial_\theta \Phi(u^{k},\theta^{k}) + \mathcal{N}_{\mathcal S}(\theta^{k})$ such that $\| \xi^k \| \le \varepsilon_{k}$. And by the convexity of $\Phi$ with respect to $\theta$, for any $\theta \in \mathcal S$,
	\[
	\Phi(u^{l},\theta^{l}) + \langle \xi^l, \theta  - \theta^{l} \rangle \le \Phi(u^{l},\theta).
	\]
	Since $\Phi$ is assumed to be continuous on $\mathcal S$, $(u^l, \theta^l) \rightarrow (\bar{u}, \bar{\theta})$ and $\xi^l \rightarrow 0$, by taking $l \rightarrow \infty$, we can obtain from the above inequality that
	\[
	\Phi(\bar{u}, \bar{\theta}) \le \Phi(\bar{u},\theta), \quad \forall \theta \in \mathcal S.
	\]
	Next, since $u^l$ satisfies
	\[
	u^l \in \mathrm{argmin}_{u \in C} \Phi(u^{l},\theta^{l-1}) + \langle \nabla_u\Phi(u^{l},\theta^{l-1}), u - u^{l} \rangle,
	\]
	it follows that, for any $j \in [p]$ and $I_u(u^l,j) := \{ i \in [n] ~|~ u^l_{i,j} = 1 \}$,
	\begin{equation}\label{thm_conv_inexact_eq1}
			 \nabla_{u_{I_u(u^l,j),j}} \Phi(u^l, \theta^{l-1}) \le \nabla_{u_{i,j}} \Phi(u^l, \theta^{l-1}), \quad \forall i \in [n].
	\end{equation}
	 Because $u^l \in C$, and thus there exists $\bar{u} \in C$ such that $u^l = \bar{u}$ for all sufficiently large $l$. 
	 By taking $l \rightarrow \infty$ in \eqref{thm_conv_inexact_eq1}, it follows from the continuity of $\Phi$, $\|\theta^{k+1}-\theta^k\| \rightarrow 0$ and $\theta^l \rightarrow \bar{\theta}$, that for any $j \in [p]$ and $I_u(\bar{u},j) := \{ i \in [n] ~|~ \bar{u}_{i,j} = 1 \}$,
	 \begin{equation}\label{}
	 	\nabla_{u_{I_u(\bar{u},j),j}} \Phi(\bar{u}, \bar{\theta}) \le \nabla_{u_{i,j}} \Phi(\bar{u}, \bar{\theta}), \quad \forall i \in [n].
	 \end{equation}
	This implies that
	\[
	\bar{u} \in  \mathrm{argmin}_{u \in C} \Phi(\bar{u}, \bar{\theta}) + \langle \nabla_u\Phi(\bar{u}, \bar{\theta}), u - \bar{u} \rangle.
	\]
	Then it follows from the assumption that $\mathrm{argmin}_{i \in [n]} \nabla_{u_{i,j}}\Phi(\bar{u}, \bar{\theta})$ is a singleton for any $j \in [p]$ and Lemma \ref{lem1} that there exists  $\epsilon_u >0$ such that
	\begin{equation}\label{eq1}
	\Phi(\bar{u}, \bar{\theta})  < \Phi(v, \bar{\theta}), \quad \forall v \in \hat{C}\cap \mathbb{B}_{\epsilon_u}(\bar{u}) \backslash \{\bar{u}\}.
	\end{equation}
	We now show that there exists $\epsilon > 0$ such that for any $\theta_\epsilon \in \mathcal{S} \cap \mathbb{B}_{\epsilon}(\bar{\theta}) $, 
	$$\bar{u}  \in \mathrm{argmin}_{u \in \hat{C} \cap \mathbb{B}_{\epsilon_u}(\bar{u}) } \Phi(u,\theta_\epsilon). $$
	If this is not the case, according to Lemma \ref{lem_r}, there must be sequence $\{(u_\ell, \theta_\ell)\}$ such that $u_\ell \in (\hat{C} \cap \mathrm{bdy}\ \mathbb{B}_{\epsilon_u}(\bar{u}) ) \cup (C\backslash\{\bar{u}\} \cap\mathbb{B}_{\epsilon_u}(\bar{u}) )$, $\theta_\ell \in \mathcal{S}$, $\theta_\ell \rightarrow \bar{\theta}$ and 
	\[
	\Phi(u_\ell, \theta_\ell) < \Phi(\bar{u},\theta_\ell).
	\]
	Since $\{u_\ell\}$ is bounded, let $\tilde{u}$ be a limit point of $\{u_\ell\}$. We have $\tilde{u} \in (\hat{C} \cap \mathrm{bdy}\ \mathbb{B}_{\epsilon_u}(\bar{u}) ) \cup (C\backslash\{\bar{u}\} \cap\mathbb{B}_{\epsilon_u}(\bar{u}) ) \subseteq \hat{C}\cap \mathbb{B}_{\epsilon_u}(\bar{u}) \backslash \{\bar{u}\}$. Taking $l \rightarrow \infty$ in above inequality, the continuity of $\Phi$ on its domain and $\theta_\ell \rightarrow \bar{\theta}$ yields that 
	\[
	\Phi(\tilde{u}, \bar{\theta}) \le \Phi(\bar{u},\bar{\theta}),
	\]
	which contradicts \eqref{eq1}. 	
	
	If there exists $(u_0, \theta_0) \in \left( \hat{C} \cap \mathbb{B}_{\epsilon_u}(\bar{u}) \right) \times \left( \mathcal{S} \cap \mathbb{B}_{\epsilon}(\bar{\theta}) \right)$ such that $\Phi(u_0, \theta_0) < \Phi(\bar{u}, \bar{\theta})$, since $\bar{u} \in \mathrm{argmin}_{u \in \hat{C} \cap \mathbb{B}_{\epsilon_u}(\bar{u})} \Phi(u,\theta_0) $, then $$\Phi(\bar{u}, \bar{\theta}) \le \Phi(\bar{u} , \theta_0) \le \Phi(u_0, \theta_0) < \Phi(\bar{u}, \bar{\theta}),$$ which arrives at a contradiction and we get the conclusion. 
\end{proof}

Because the functions $f_{i,j}(\theta)$ usually vary distinctly for each $i, j$ in the problem \eqref{prb2} derived from \eqref{EnergyGeneral}, we introduce the following mild assumption.

\begin{assumption}\label{ass2}
	For any $u \in C$, $\theta \in \mathcal S$, $j \in [p]$ and $i_1, i_2 \in [n]$, if $i_1 \neq i_2$, then $	\nabla_{u_{i_1,j}}\Phi({u}, {\theta}) \neq  	\nabla_{u_{i_2,j}}\Phi({u}, {\theta})$.
\end{assumption}

With Assumption \ref{ass2} fulfilled, the assumption that $\mathrm{argmin}_{i \in [n]} \nabla_{u_{i,j}}\Phi(\bar{u}, \bar{\theta})$ is a singleton for any $j \in [p]$ in Theorem \ref{thm_conv_inexact} holds. Consequently, we can straightforwardly derive the following result from Theorem \ref{thm_conv_inexact}.

\begin{corollary}\label{coro1}
Suppose Assumption \ref{ass2} is satisfied. Let $\{(u^k, \theta^k)\}$ be the sequence generated by alternating minimization Algorithm \ref{alg3} that satisfies condition \eqref{iAM2} with  $\varepsilon_k \rightarrow 0$.
Consider any limit point  $(\bar{u}, \bar{\theta})$ of sequence $\{(u^k, \theta^k)\}$. Then $\bar{u} \in C$.
Suppose that $\|\theta^{k+1}-\theta^k\| \rightarrow 0$, 
then $(\bar{u}, \bar{\theta})$  is a local minimizer of problem \eqref{prb2} in the sense that there exists $\epsilon > 0$ such that
\[
\Phi(\bar{u}, \bar{\theta}) \le \Phi(u, \theta), \quad \forall (u,\theta) \in \left(  \hat{C} \cap \mathbb{B}_{{\epsilon}}(\bar{u})  \right) \times \left( \mathcal{S} \cap \mathbb{B}_{\epsilon}(\bar{\theta}) \right) .
\]
\end{corollary}

\subsubsection{Analysis  on updating $\theta$ with one-step acceleration}

Finding the approximate stationary point during $\theta$  update in Algorithm \ref{alg3} can be challenging due to the vanishing accuracy $\{\varepsilon_k\}$. We therefore propose a simple acceleration scheme for the $\theta$ update
given that $\Phi$ satisfies the following Assumption \eqref{assum_smooth}. 
\begin{assumption}\label{assum_smooth}
$\Phi$ is a continuously differentiable and $\nabla_\theta \Phi$ is $L_{\theta}$-Lipschitz continuous.
\end{assumption}

The idea is,  during $\theta$ update in Algorithm \ref{alg3}, to approximate the stationary solution to problem $$\mathrm{min}_{\theta \in \mathcal S} \Phi(u^{k+1},\theta)$$
by using only a single projected gradient  iteration step, without solving the above subproblem until the accuracy $\{\varepsilon_k\}$ is met as in \eqref{iAM2}. The iterative procedure is outlined in Algorithm~\ref{alg4}.

\begin{algorithm}[ht!]
	\begin{flushleft}
        {\bf Input:}{ $(u^0, \theta^0) \in C \times \mathcal{S}$, $\alpha_\theta \in (0,2/L_{\theta})$ }
        
	{\bf Output:} { ($\bar u, \bar \theta$)  that approximately minimizes  \eqref{prb2} .}
	
	Set $k=0$\;

{\bf While} { $\|u^{k+1} - u^k\| + \|\theta^{k+1} - \theta^k\| > tol$
	}
	
	{\bf 1.} Find $u^{k+1} \in C$  by using \eqref{th_u} with $ u^{k,0} =  u^{k}$ until $u^{k,T+1}  = u^{k,T}$. Set $u^{k+1} = u^{k,T}$.
				
            {\bf 2.} Update $\theta^{k+1}$ by a single projected gradient step as
		\begin{equation}\label{subprob2_pg}
			\theta^{k+1} = \mathrm{Proj}_{\mathcal{S}} \left( \theta^k - \alpha_\theta \nabla_\theta \Phi(\theta^k,u^{k+1}) \right).
		\end{equation}
		
		Set $k = k+1$\;
\end{flushleft}
	\caption{An one-step acceleration version of  Algorithm \ref{alg3} }  \label{alg4}
\end{algorithm}

\begin{theorem}\label{thm_min3}
Suppose Assumption \ref{assum_smooth} is satisfied and $\Phi$ is bounded below on $\hat{C}\times \mathcal{S}$. Let $\{(u^k, \theta^k)\}$ be the sequence generated by alternating minimization Algorithm \ref{alg4}. 
Consider any limit point  $(\bar{u}, \bar{\theta})$ of sequence $\{(u^k, \theta^k)\}$. Then $\bar{u} \in C$.
Suppose that $\mathrm{argmin}_{i \in [n]} \nabla_{u_{i,j}}\Phi(\bar{u}, \bar{\theta})$ is a singleton for any $j \in [p]$, or Assumption \ref{ass2} is satisfied.
Then $(\bar{u}, \bar{\theta})$ is a local minimizer of problem \eqref{prb2} in the sense that there exists $\epsilon > 0$ such that
\[
\Phi(\bar{u}, \bar{\theta}) \le \Phi(u, \theta), \quad \forall (u,\theta) \in \left(  \hat{C} \cap \mathbb{B}_{\epsilon}(\bar{u})  \right) \times \left( \mathcal{S} \cap \mathbb{B}_{\epsilon}(\bar{\theta}) \right) .
\]
\end{theorem}
\begin{proof}
Since $\Phi$ is convex with respect to $\theta$ and $\nabla_\theta \Phi$ is $L_{\theta}$-Lipschitz continuous, by the celebrated sufficient decrease lemma for convex function, we have that $$\Phi(u^{k+1},\theta^{k+1}) \le \Phi(u^{k+1},\theta^{k})- (\frac{1}{\alpha_{\theta}} - \frac{L_{\theta}}{2})\|\theta^{k+1} - \theta^k\|^2.$$ Combining with the fact that $\Phi(u^{k+1}, \theta^k) \le \Phi(u^k, \theta^k)$ yields that 
\[
\Phi(u^{k+1},\theta^{k+1}) \le \Phi(u^{k},\theta^{k}) - (\frac{1}{\alpha_{\theta}} - \frac{L_{\theta}}{2})\|\theta^{k+1} - \theta^k\|^2.
\]
As $\Phi$ is assumed to be bounded below on $\hat{C}\times \mathcal{S}$, we have 
\[
\lim_{k\rightarrow \infty} \|\theta^{k+1} - \theta^k\| = 0.
\]
Next, since 
\[
\nabla_\theta \Phi(\theta^{k+1},u^{k+1}) -\nabla_\theta \Phi(\theta^k,u^{k+1}) +  \frac{1}{\alpha_\theta}(\theta^{k+1} - \theta^k) \in \nabla_\theta \Phi(u^{k+1},\theta^{k+1}) + \mathcal{N}_{\mathcal S}(\theta^{k+1}),
\]
we have $$\mathrm{dist}\left(0, \nabla_\theta \Phi(u^{k+1},\theta^{k+1}) + \mathcal{N}_{\mathcal S}(\theta^{k+1}) \right) \rightarrow 0$$ and then the conclusion follows from Theorem \ref{thm_conv_inexact} and Corollary \ref{coro1} immediately.
\end{proof}

\subsubsection{Interpretation for the theoretical validity of ICTM}

While people are not currently aware of the convergence guarantees for ICTM towards solutions with good qualities, in practice it is usually able to reach a local minimizer. 
We shall provide an interpretation for the theoretical validity of ICTM
from the perspective of $u$ update in \eqref{th_u}.
Recall that in the ICTM, an important step is to adopt a simple thresholding to update $u$. The precise scheme is based on finding an approximate solution via solving the minimization of the linearization of $\Phi(u,\theta^{k})$ at $u^k$, 
\[\min_{u\in \hat C} L_{u^k}(u) : = \Phi(u^k,\theta^{k})+ \sum_{i=1}^n\sum_{j=1}^p (u_{i,j}-u_{i,j}^k) {\nabla_{u_{i,j}} \Phi(u^k,\theta^k)}\]
which is further equivalent to
\[\min_{u\in \hat C}  \sum_{i=1}^n\sum_{j=1}^p u_{i,j} {\nabla_{u_{i,j}} \Phi(u^k,\theta^k)}.\]

This can be done for each $j$ independently because of the constraints $\sum_{i=1}^n u_{i,j} = 1$ and $u_{i,j}\geq 0$ as defined in $\hat C$. In this regard, the update of $u_{i,j}$ for a fixed $j$ in the ICTM reads  as
\begin{align}\label{thresholding}
u_{i,j} = 
\begin{cases}
1 & \textrm{if} \ \ i = \min\left\{\arg\min_m {\nabla_{u_{i,j}} \Phi(u^k,\theta^k)}\right\}, \\
0 & \textrm{otherwise.}
\end{cases}
\end{align}

\subsection{Summary of the convergence results}
{To the end of this section,
we summarize our convergence results for the three cases discussed above in following:
\begin{itemize}
	\item [1.] In global minimizer for $\theta$-update and global minimizer for $u$-update case, alternating minimization scheme generates a local minimizer $(\bar{u},\bar{\theta})$ of problem \eqref{prb2} in the sense that there exists $\epsilon > 0$ such that
	\[
	\Phi(\bar{u}, \bar{\theta}) \le \Phi(u, \theta), \quad \forall (u,\theta) \in \hat{C} \times (\mathcal S \cap \mathbb{B}_{\epsilon}(\bar{\theta})).
	\]
	\item [2.] In global minimizer for $\theta$-update and local minimizer for $u$-update case, alternating minimization scheme generates a local minimizer $(\bar{u},\bar{\theta})$ of problem \eqref{prb2} in the sense that there exists $\epsilon > 0$ such that
	\[
	\Phi(\bar{u},\bar{\theta}) \le \Phi(u, \theta), \quad \forall (u,\theta) \in \left({C} \cap \mathbb{B}_r(\bar{u}) \right)\times (\mathcal S \cap \mathbb{B}_{\epsilon}(\bar{\theta})).
	\]
	\item [3.] In approximate stationary solution for $\theta$-update and thresholding for $u$-update case, alternating minimization scheme generates a local minimizer $(\bar{u},\bar{\theta})$ of problem \eqref{prb2} in the sense that there exists $\epsilon > 0$ such that
	\[
	\Phi(\bar{u}, \bar{\theta}) \le \Phi(u, \theta), \quad \forall (u,\theta) \in \left(  \hat{C} \cap \mathbb{B}_{\epsilon}(\bar{u})  \right) \times \left( \mathcal{S} \cap \mathbb{B}_{\epsilon}(\bar{\theta}) \right) .
	\]
\end{itemize}
}

\section{Numerical experiments} \label{sec:exp}

In this section, we illustrate the performance of the algorithms for different examples via several numerical examples. We implemented the algorithms in MATLAB. All reported results were obtained on a laptop with a 3.8GHz Intel Core i7 processor and 8GB of RAM. We apply our methods to different models and numerical results show a clear advantage of the method in terms of simplicity and efficiency. 

\subsection{Application to the implicit surface reconstruction problem}\label{sec:surface} 

In the first example, we consider the following implicit surface reconstruction model from a point cloud,

\begin{align}
\min_{u(x)\in\{0,1\}} \mathcal E_{ISC}(u) \colon =\sqrt{\frac{\pi}{\tau}}\int_{\mathbb R^n}  |d|^{\frac{p}{2}} u \ G_\tau * \left(|d|^{\frac{p}{2}} (1-u)\right) \ dx.\label{eq:approx1}
\end{align}

In this model, $d(x)$ is the distance function from arbitrary points in the space to the point cloud, $u(x)$ is the decision variable which is an indicator function to implicitly determine the inside and outside of the reconstructed surface. Because the point cloud is fixed for the reconstruction, $d(x)$ is then fixed and the only decision variable in this case is $u$. 

After discretization, the problem reads, 
\begin{align} 
\min_{u \in \{0,1\}^p}  \Psi(u) =  [\tilde d \circ (1-u)] K_\tau (\tilde d  \circ u)^T \label{eq:isc}
\end{align}
where $u =(u_1,u_2, \cdots, u_p)$ is the indicator function at $p$ discrete points in the computational domain, $K_\tau$ is a symmetric positive definite matrix comes from the discretization and rearrangement of scaled Gaussian kernel, $\circ$ represents the Hadamard product, and $\tilde d$ is the distance from discrete points in the computational domain to the point cloud (see Figure~\ref{fig:pointd}). We refer more details on the model and related work to \cite{Wang_2021}.

\begin{figure}[ht!]
\centering
\includegraphics[width = 0.5\textwidth, clip,trim = 6cm 6cm 6cm 5cm]{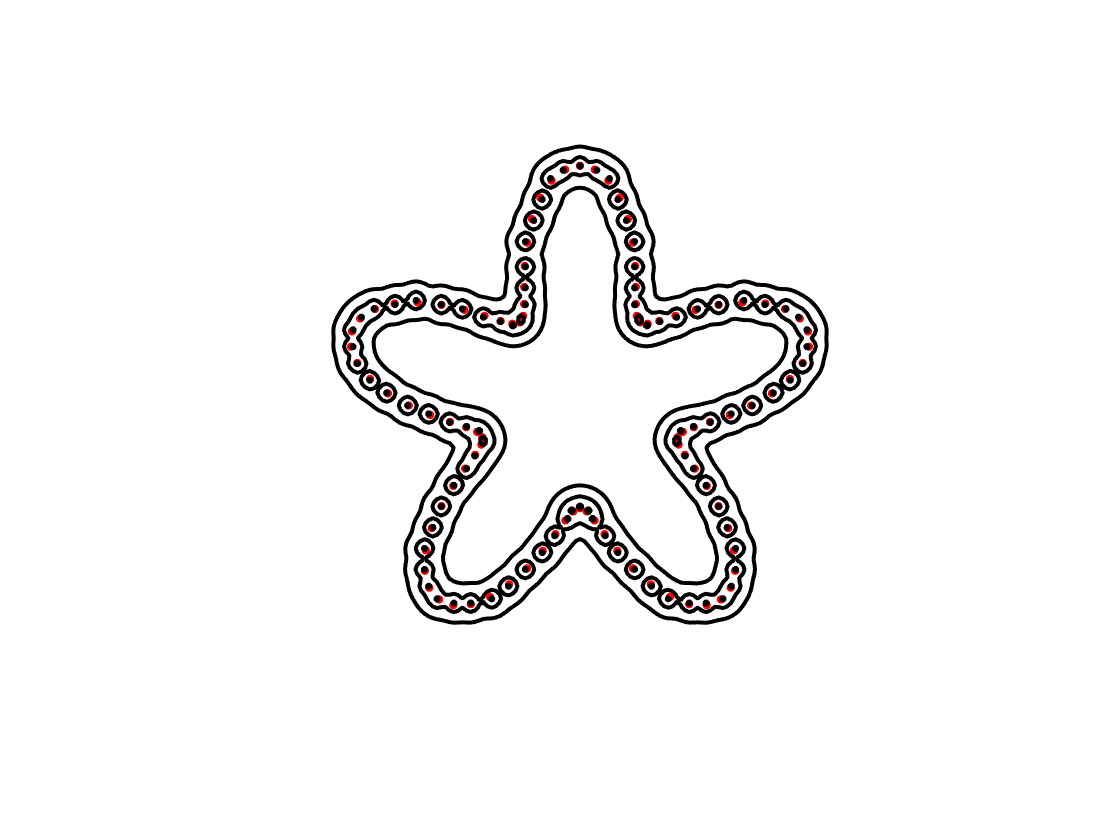}
\caption{A diagram to several level sets of the distance function $\tilde d$ to a discrete point cloud. See Section~\ref{sec:surface}.}\label{fig:pointd}
\end{figure}

It is clear that the problem \eqref{eq:isc} is concave with respect to $u$, we simple apply the thresholding step as that in \eqref{th_u}  to general 2- and 3-dimensional point clouds as displayed in Figure~\ref{fig:isc3}:

\begin{align} \label{eq:iscthresholding}
u_i^{k+1} =  
\begin{cases}
1 & \textrm{if} \ \ (K_\tau (\tilde d  \circ (1-2u))^T)_i <0\\
0 & \textrm{otherwise}
\end{cases}.
\end{align}

All converge in fewer than 100 iteration steps to the minimizer, which implies the efficiency of the algorithm and is consistent with our analysis.

\begin{figure}[ht!]
\centering
\includegraphics[width = 0.24\textwidth, clip,trim = 6cm 3cm 5cm 3cm]{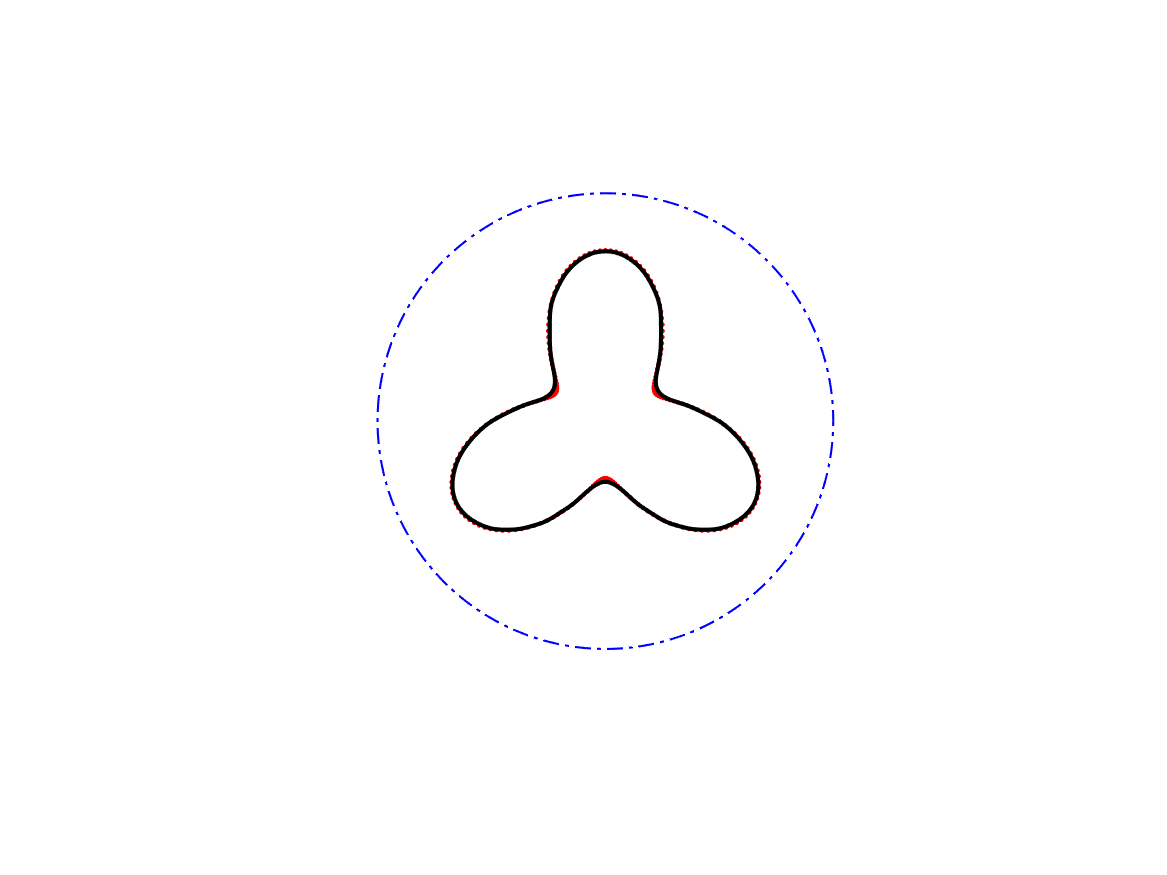}
\includegraphics[width = 0.24\textwidth, clip,trim = 6cm 3cm 5cm 3cm]{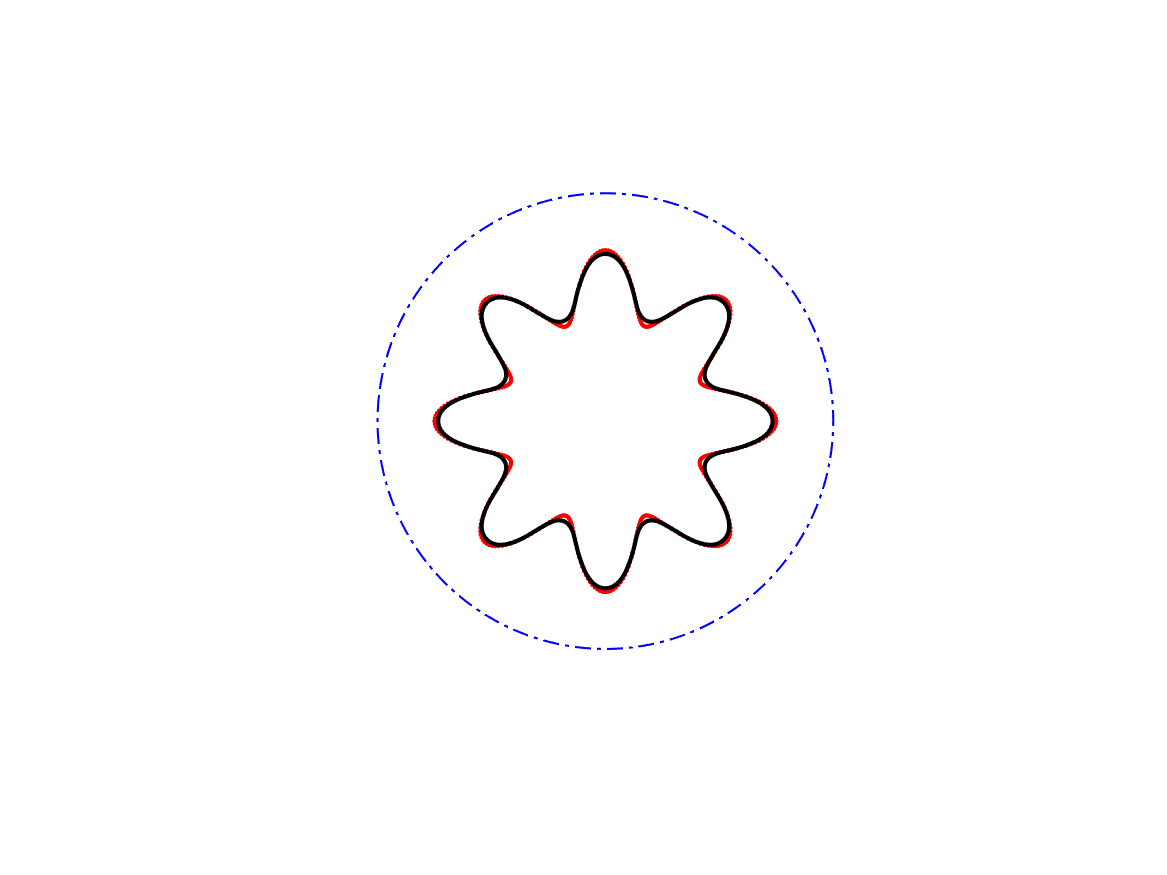}
\includegraphics[width = 0.24\textwidth, clip,trim = 6cm 3cm 5cm 3cm]{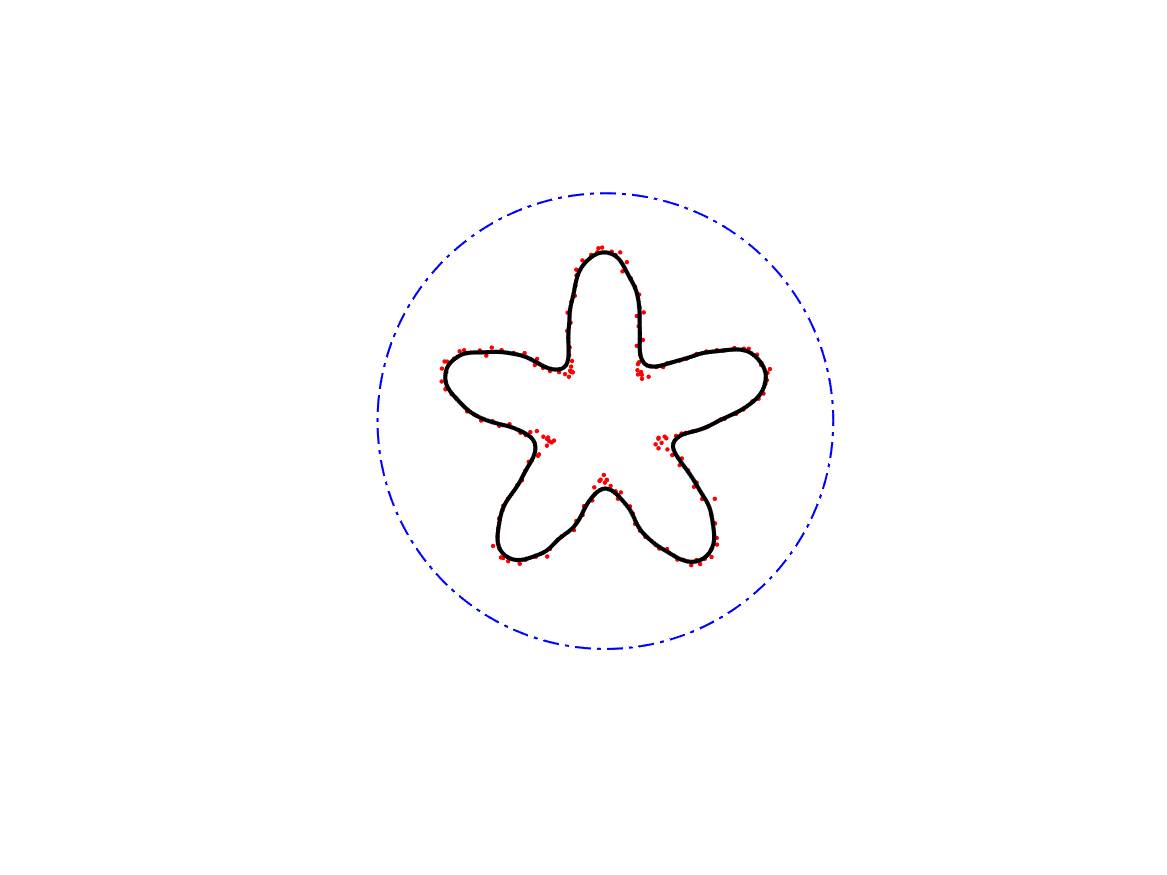}
\includegraphics[width = 0.24\textwidth, clip,trim = 6cm 3cm 5cm 3cm]{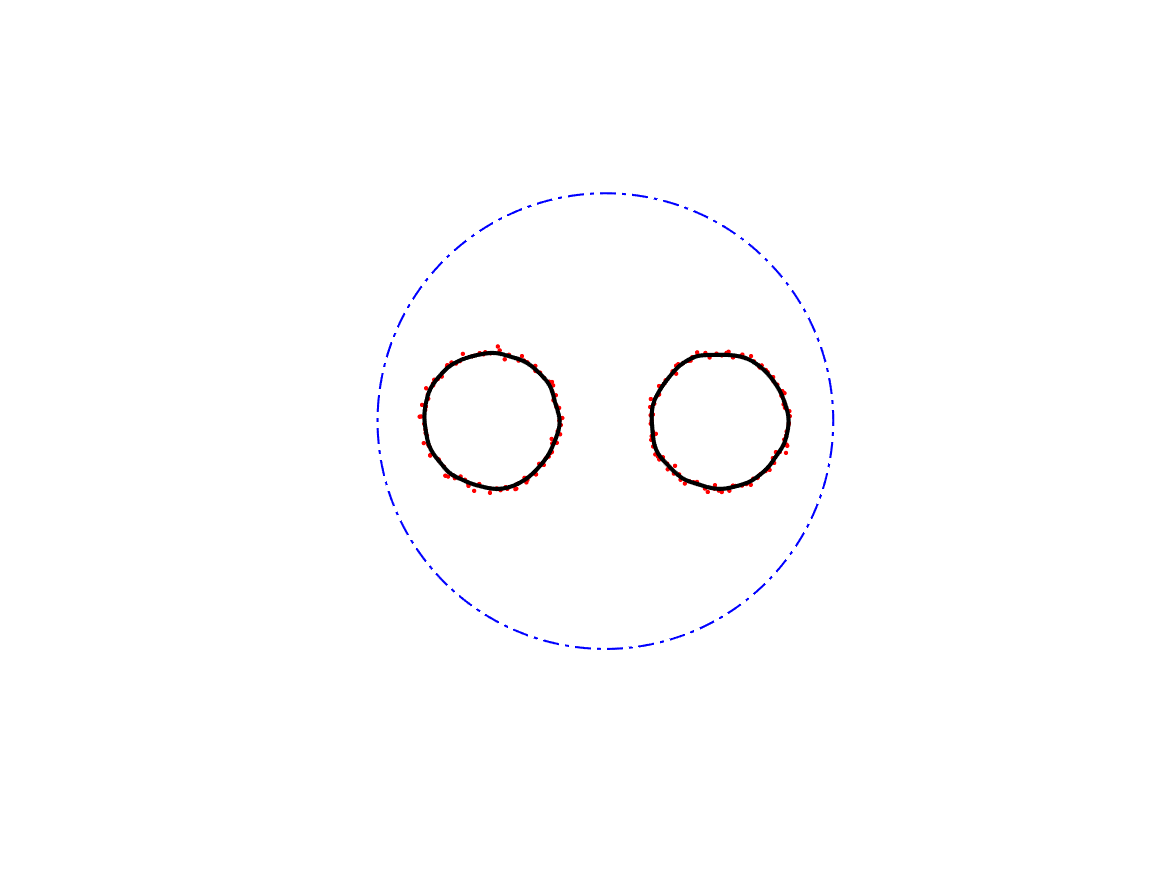}\\
\includegraphics[width = 0.24\textwidth, clip,trim = 4cm 2cm 4cm 2cm]{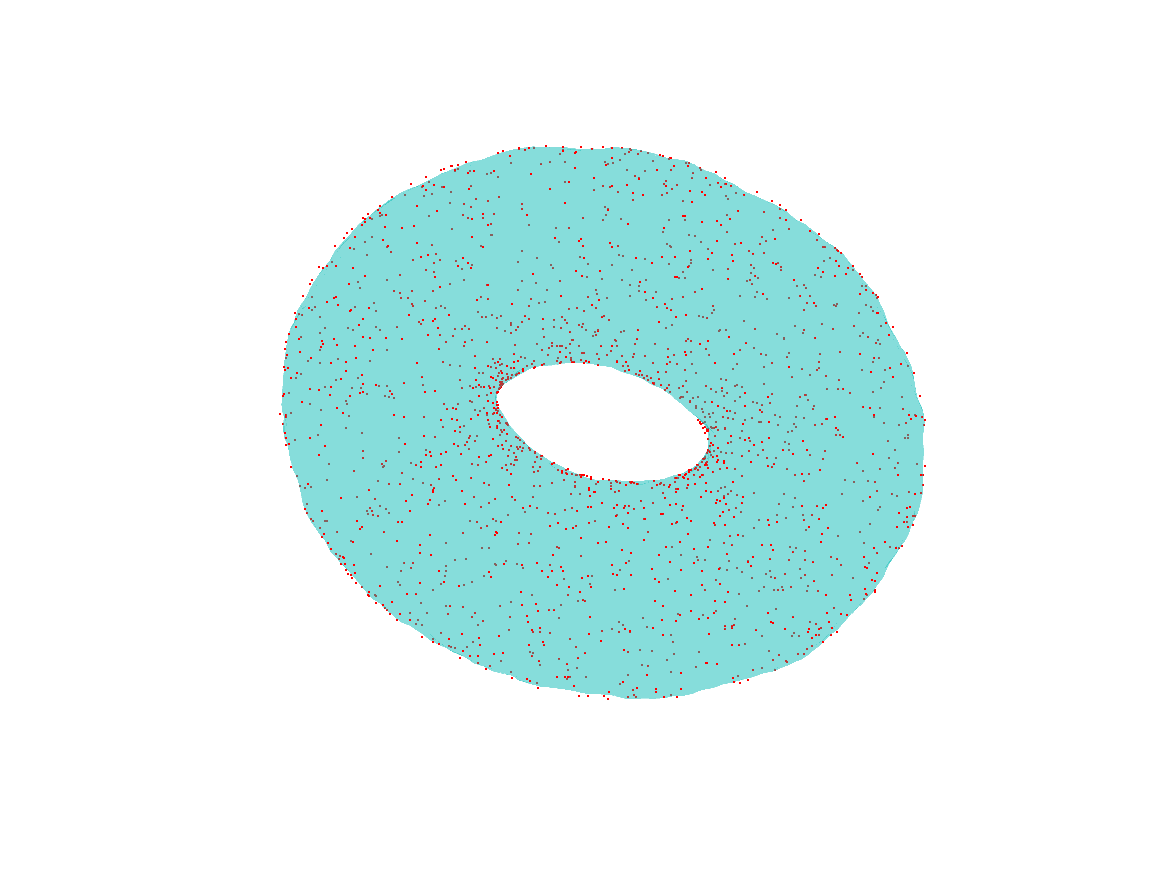}
\includegraphics[width = 0.24\textwidth, clip,trim = 6cm 4cm 6cm 2cm]{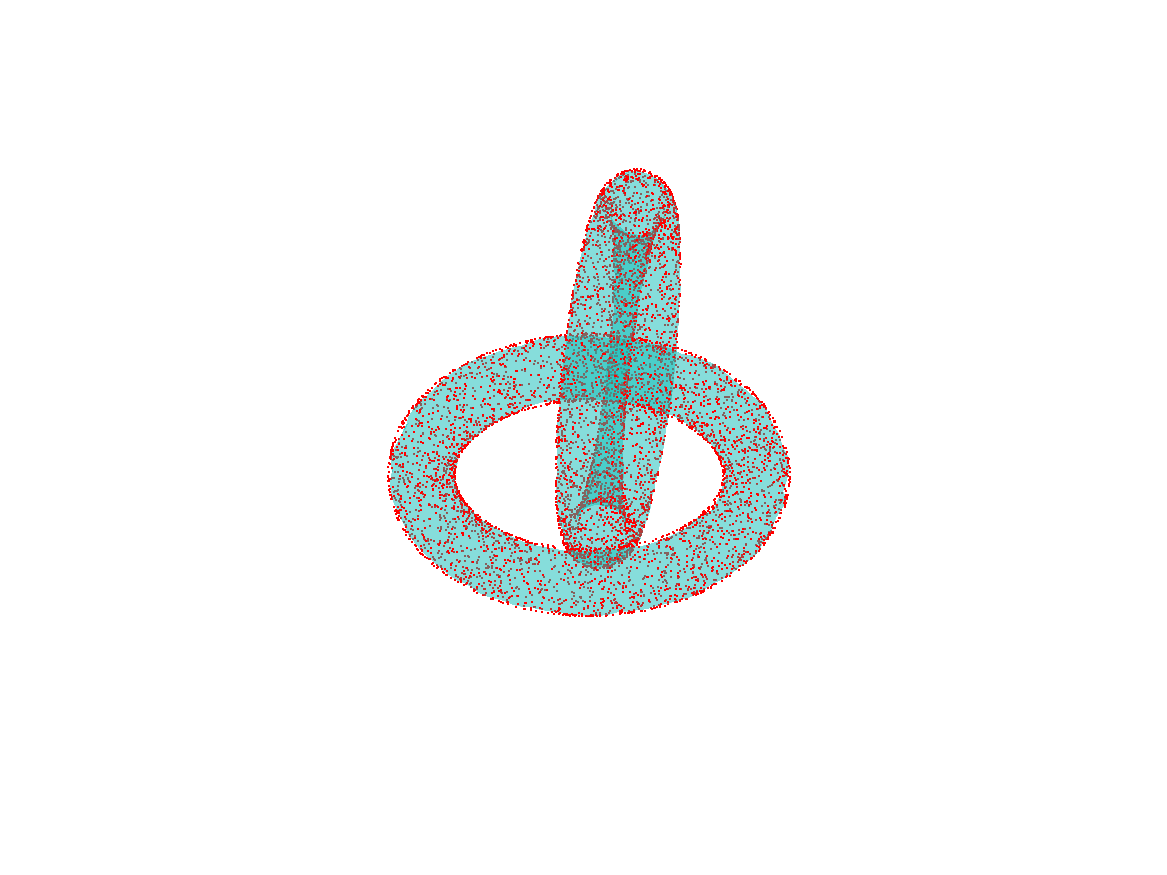}
\includegraphics[width = 0.24\textwidth, clip,trim = 5cm 2cm 4cm 2cm]{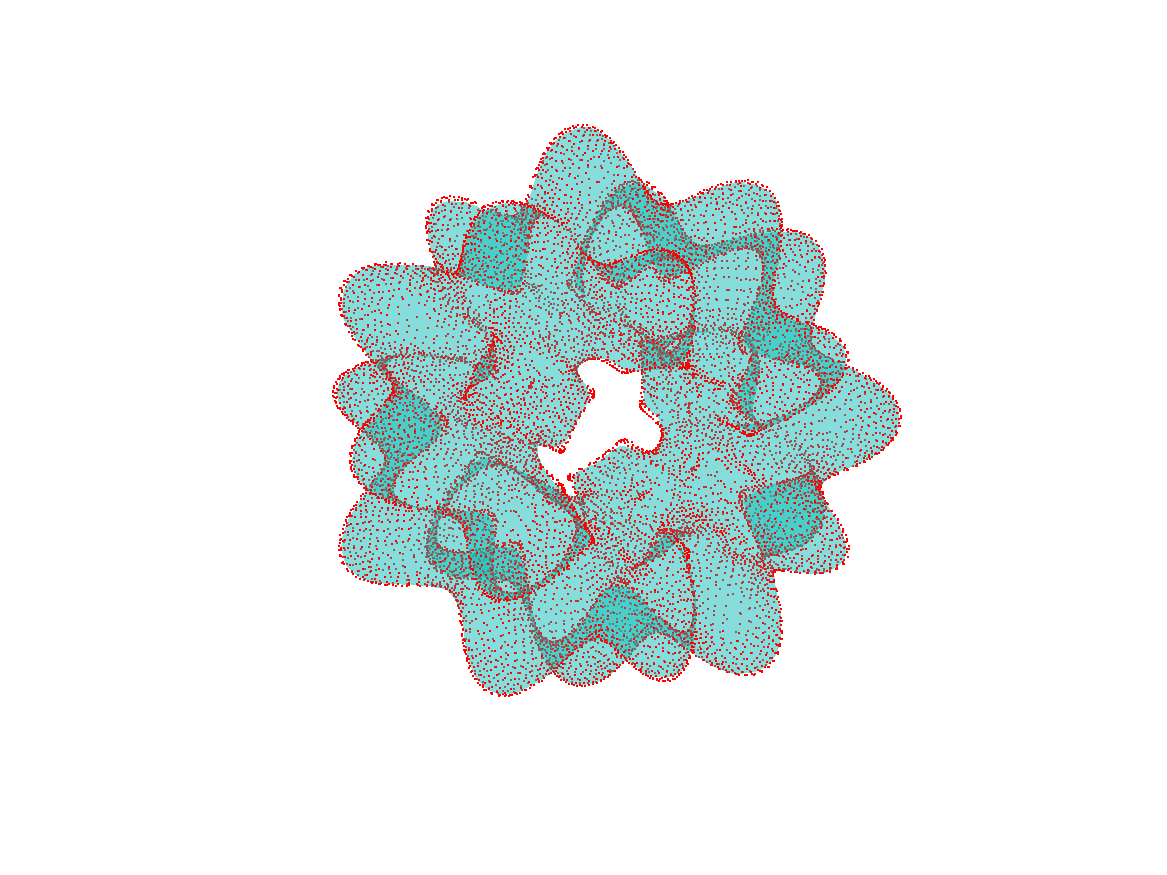}
\includegraphics[width = 0.24\textwidth, clip,trim = 5cm 2cm 4cm 2cm]{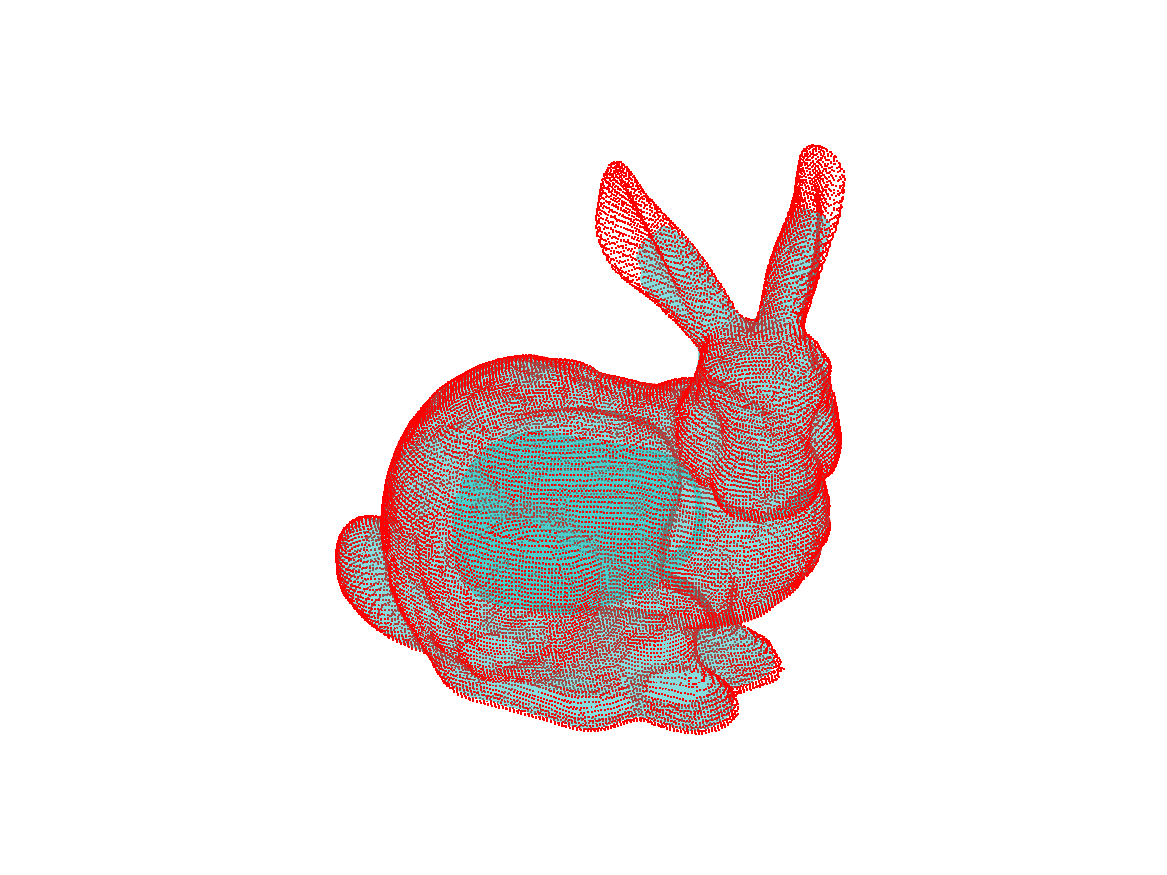}
\caption{Results obtained from \eqref{eq:iscthresholding} for different point clouds. See Section~\ref{sec:surface}.}\label{fig:isc3} 
\end{figure}

{Furthermore, we use several simple examples to compare the efficiency between the method and level set approaches. A recent work in \cite{he2019fast} has carefully studied the efficiency comparison between the semi-implicit method (SIM) and the classic level set approach in \cite{Zhao_2000} and showed a big improvement in the efficiency. Therefore, we simply compare the efficiency between our method and SIM in \cite{he2019fast}. 

We consider different point clouds (different $m$) generated using $N=200$ uniform points $\theta_i$ in $[0,2\pi]$:
\[\begin{cases}
x_i = r_i \cos(\theta_i), \\
y_i = r_i \sin(\theta_i),
\end{cases}
\]
where $r_i = 1+0.4\sin(m\theta_i)$ with different $m = 3, 4, 5, 6, 7$, and $8$. We use the same initial guess as shown in Figure~\ref{fig:com} and same discretization ($128 \times 128$ grids) of the computational domain for two methods. The table at the bottom of  Figure~\ref{fig:com} shows the dramatical acceleration in the computational CPU time and figures indicate the improvement on the accuracy.  

\begin{figure}[ht!]
\centering
\includegraphics[width = 0.3\textwidth,clip, trim = 4cm 2cm 4cm 1.5cm]{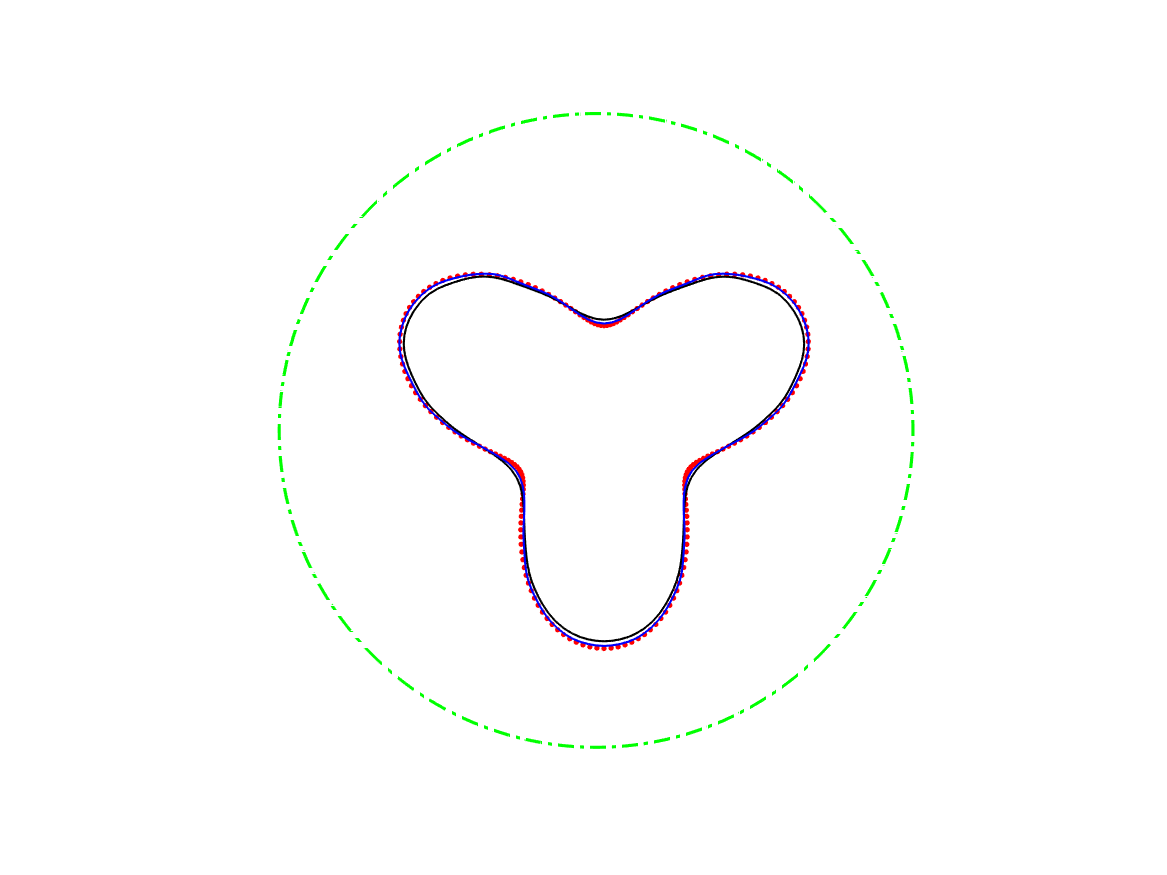}
\includegraphics[width = 0.3\textwidth,clip, trim = 4cm 2cm 4cm 1.5cm]{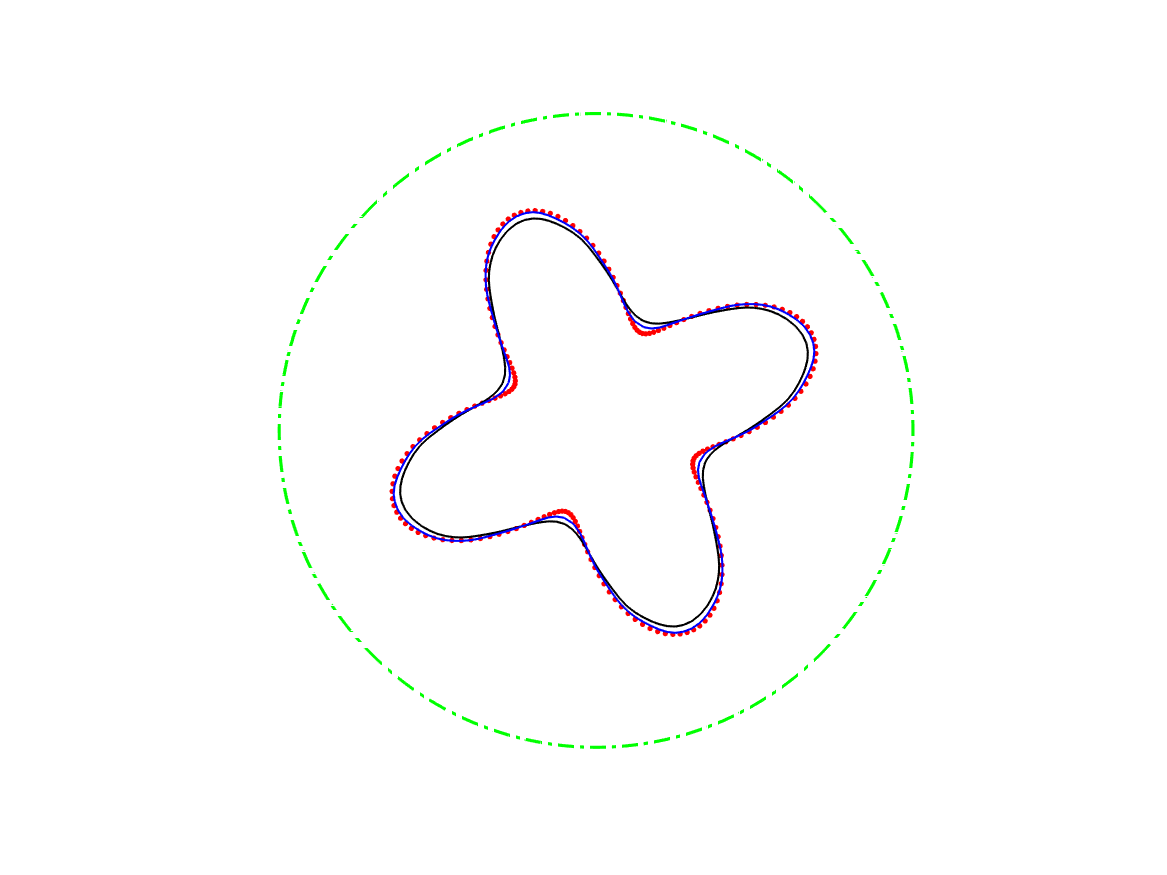}
\includegraphics[width = 0.3\textwidth,clip, trim = 4cm 2cm 4cm 1.5cm]{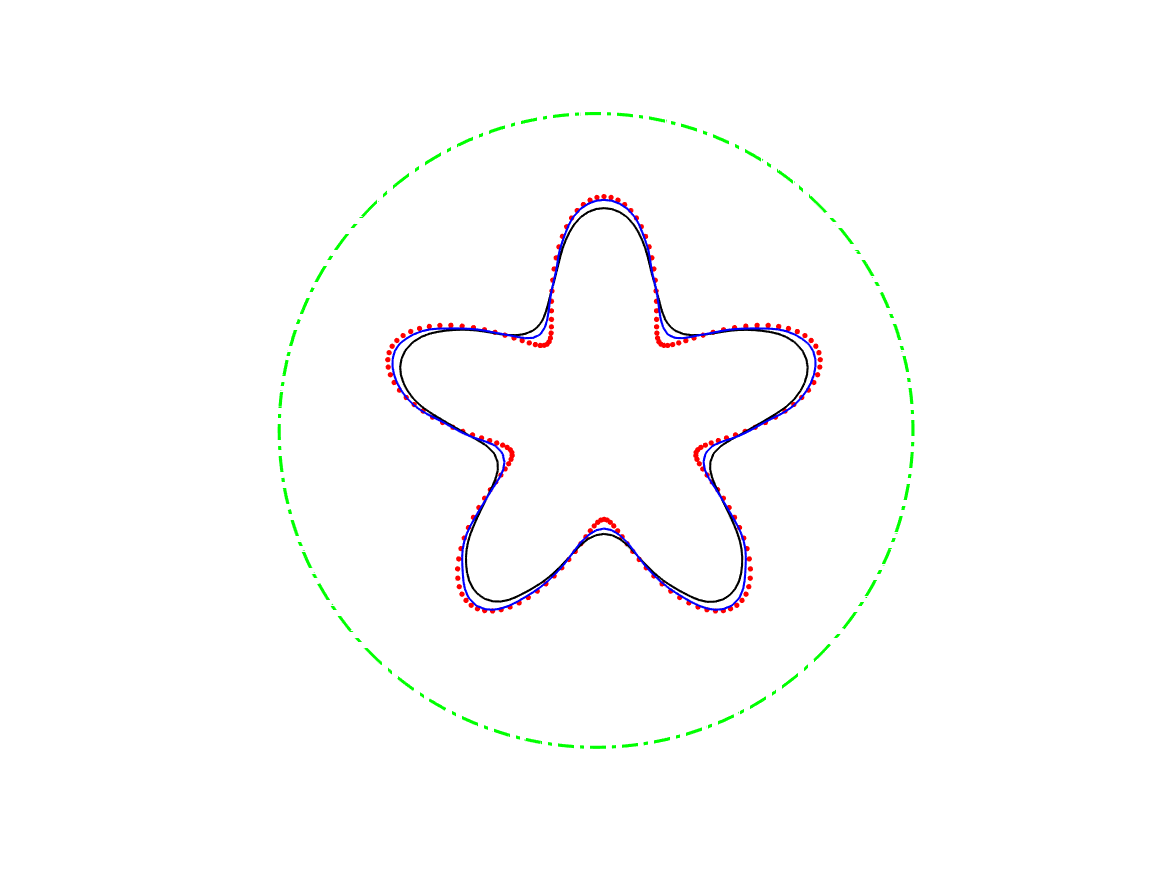}
\includegraphics[width = 0.3\textwidth,clip, trim = 4cm 2cm 4cm 1.5cm]{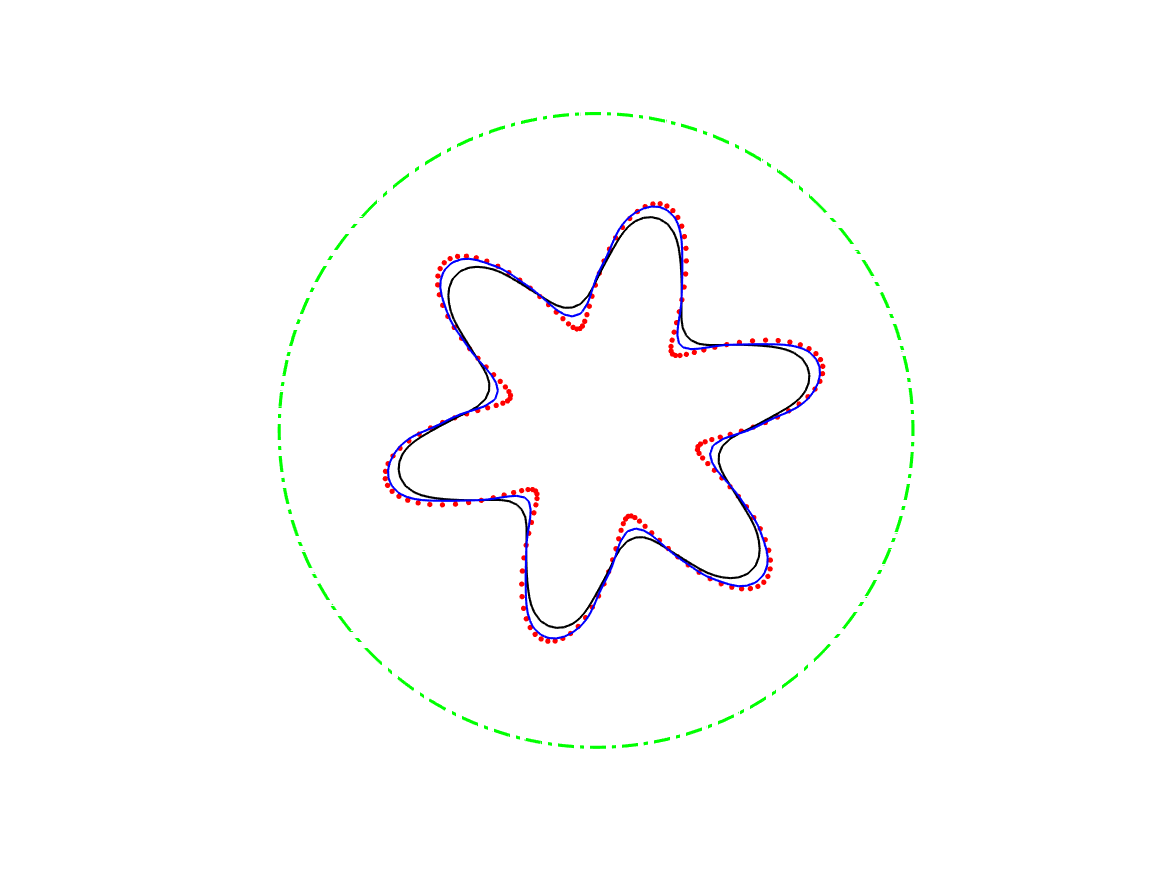}
\includegraphics[width = 0.3\textwidth,clip, trim = 4cm 2cm 4cm 1.5cm]{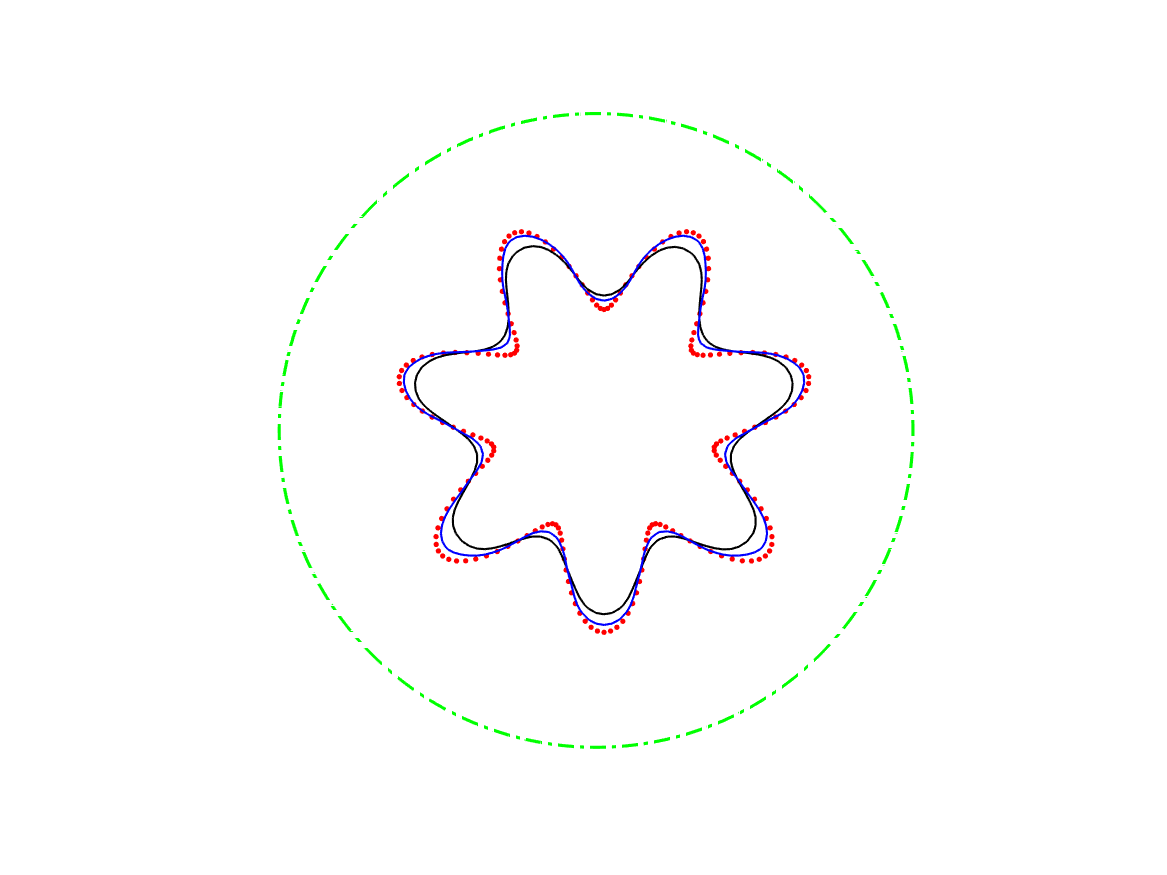}
\includegraphics[width = 0.3\textwidth,clip, trim = 4cm 2cm 4cm 1.5cm]{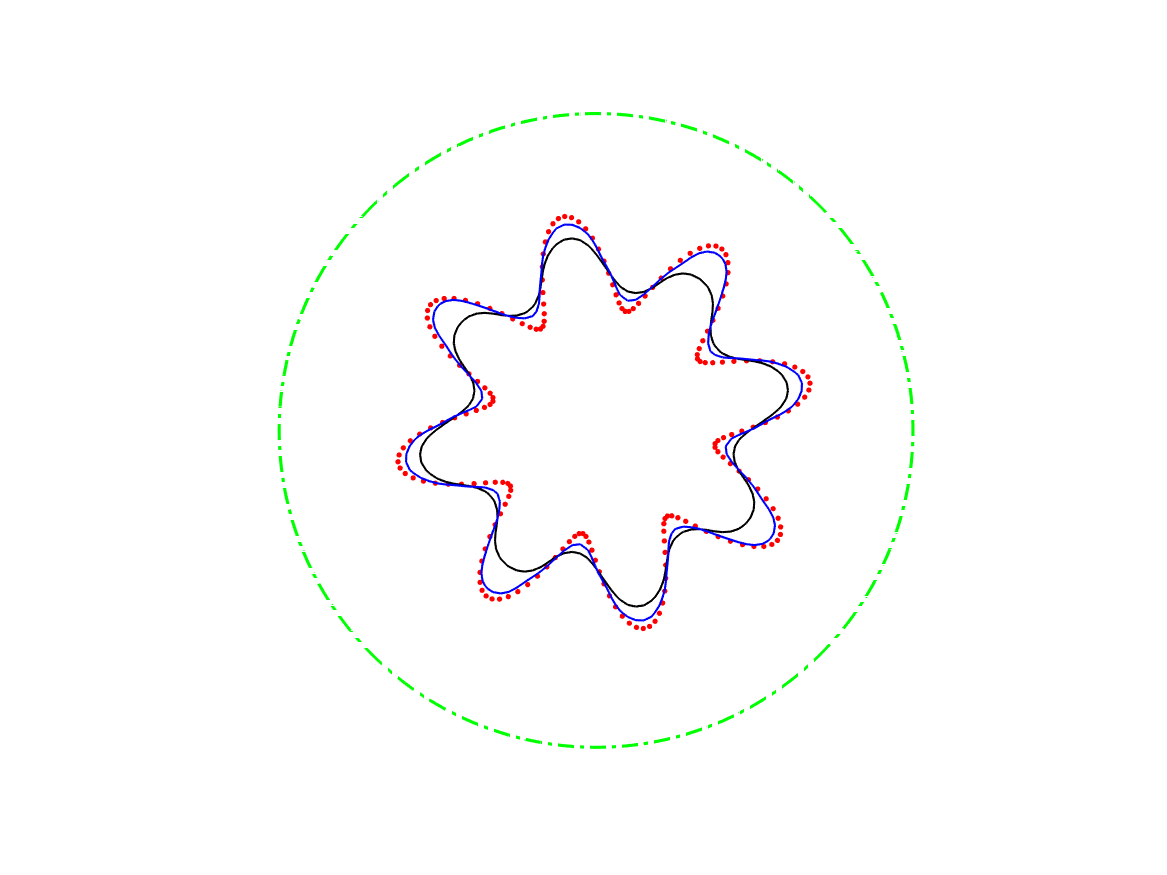}

\medskip
{\begin{tabular}{c|c|c|c|c|c|c}
SIM in \cite{he2019fast} & $3.2$ s &  $4.47$ s & $5.26$ s & $6.40$ s & $7.18$ s & $4.56$ s  \\
\hline
\hline
Our method  & $0.029$ s & $0.036$ s & $0.029$ s & $0.026$ s & $0.031$ s & $0.030$ s
\end{tabular}
\caption{Comparisons between our method and SIM in \cite{he2019fast}. Blue curve: results obtained from our method. Black curve: results obtained from SIM in \cite{he2019fast}. Green curve: initial guess.  Red points: point cloud with $200$ points. Table: CPU times for the computations. See Section~\ref{sec:surface}.} \label{fig:com}
}
\end{figure}
}

\subsection{Application to the Chan--Vese model for image segmentation}\label{sec:cv}

In this experiment, we consider the well-known Chan--Vese (CV) model \cite{Chan_2001} for image segmentation approximated using characteristic functions. Specifically, in the CV model, the objective functional for the n-segment case is 
\begin{align}
\mathcal{E}_{CV}(\Omega_1, \ldots, \Omega_n, \theta_1, \ldots, \theta_n) =\lambda  \sum_{i=1}^n |\partial \Omega_i|+\sum_{i=1}^n \int_{\Omega_i}|\theta_i-I|^2 \ dx \label{CV} 
\end{align}
where $\partial \Omega_i$ is the boundary of the $i$-th segment $\Omega_i$, $|\partial \Omega_i|$ denotes the perimeter of the domain $\Omega_i$, $I: \Omega \rightarrow [0,1]^d$ is the image, $d$ is the channel of the image ({\it e.g.}, $d=1$: gray image, $d=3$: RGB image), and $\lambda$ is a positive parameter.

Using the approximation to $|\partial \Omega_i|$ by using characteristic functions, $u_i(x)$, of each $\Omega_i$, the approximate objective functional is written into 

\begin{align}
\mathcal{E}_{CV}^\tau = \lambda \sqrt{\frac{\pi}{\tau}} \sum\limits_{i=1}^n \int_{\Omega} u_i G_{\tau} *(1-u_i) \  dx + \sum_{i=1}^n \int_{\Omega} u_i |\theta_i-I|^2 \ dx. \label{approxCV}
\end{align}

After discretization using the information on each pixel, the discretized problem is then written into 

\begin{align}
\min_{u \in C, \theta \in \mathbb R^n}  \Psi(u,\theta) = \sum_{i=1}^n \sum_{j = 1}^p u_{i,j} |\theta_i-I_j|^2 + \lambda \sum_{i=1}^n  (1-u_i) K_\tau u_{i}^T
\end{align}
where $u_i =(u_{i,1},u_{i,2}, \cdots, u_{i,p})$ here is the $i$-th row in $u$, $K_\tau$ is a symmetric positive definite matrix comes from the discretization and rearrangement of scaled Gaussian kernel, $p$ here is the number of pixels of the image, and $\theta_i$ $(i = 1,2,\cdots, n)$ is the state variable for each segment. (More details can be referred to \cite{wang2016efficient}.)

Because $\Psi(u,\theta)$ is strictly convex with respect to $\theta_i$, subproblem \eqref{subproblem1} can be explicitly given by solving 
\[\dfrac{\partial \Psi(u^k,\theta)}{\partial \theta_i} = 0\]
which is 
\[ \theta_i^{k+1} = \dfrac{ \sum_{j = 1}^p u^k_{i,j} I_{j}}{ \sum_{j = 1}^p u^k_{i,j}}.\]

We then apply Algorithm~\ref{alg3} to the above minimization problem and the subproblem \eqref{subproblem2} for $u_{i,j}$ for each $j = 1, 2, \cdots, p$ is iteratively solved by \eqref{thresholding};
\begin{align}
u^{k+1}_{i,j} = \begin{cases}
1 & \textrm{if} \ \ i = \min\left\{\arg\min_m  |\theta_m^k-I_{m,j}|^2+\lambda \left(K_\tau{(1-2u_m^k)}^T\right)_{j} \right\}, \\
0 & \textrm{otherwise.}
\end{cases}
\end{align}

In the practical implementation, the matrix multiplication $K_\tau {(1-2u_m^k)}^T$ is computed effectively by fast Fourier transform (FFT). In the follows, we apply the above algorithm into the segmentation on both gray images and colorful images. In Figure~\ref{fig1}, we list the initial guesses ($u^0$) for different images with different number of segments in the left, the final converged solution in the middle, and the decaying curve of the objective function value with respect to the iteration steps. $\lambda$ are set to be $0.03$ for the gray image and $0.1$ for the RGB image.

\begin{figure}[t!]
\centering
\includegraphics[width = 0.3\textwidth]{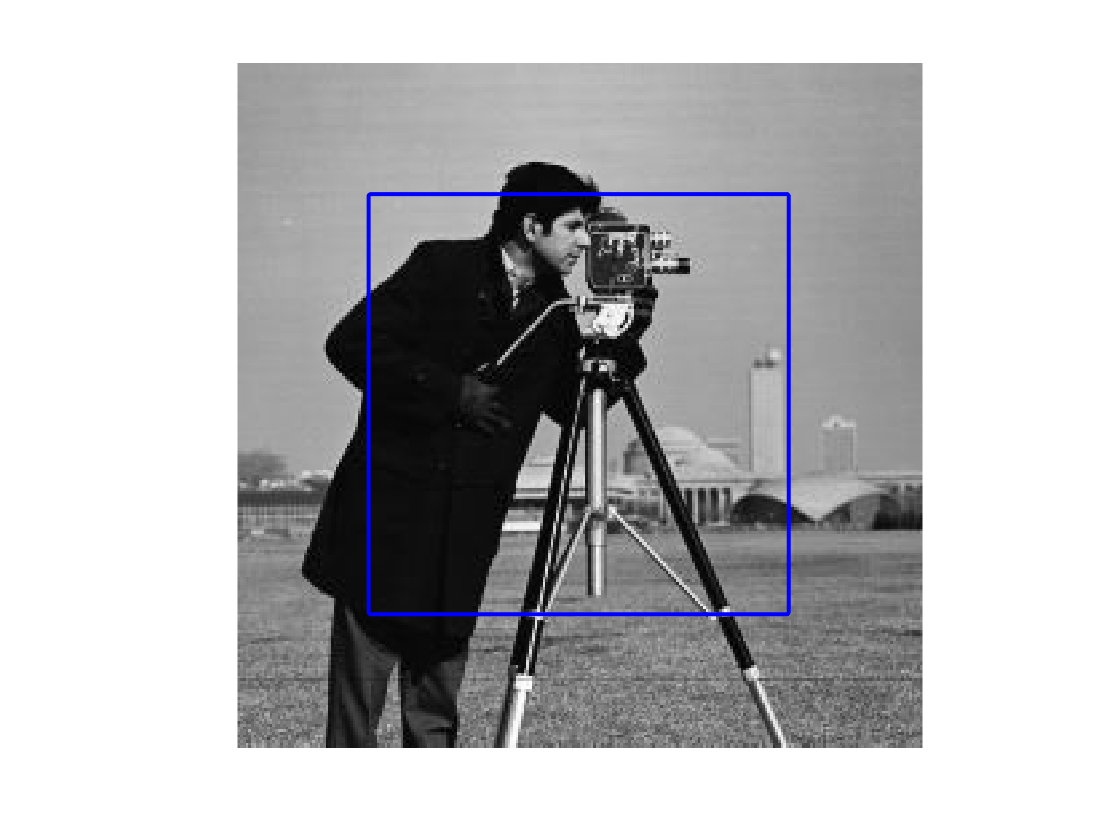}
\includegraphics[width = 0.3\textwidth]{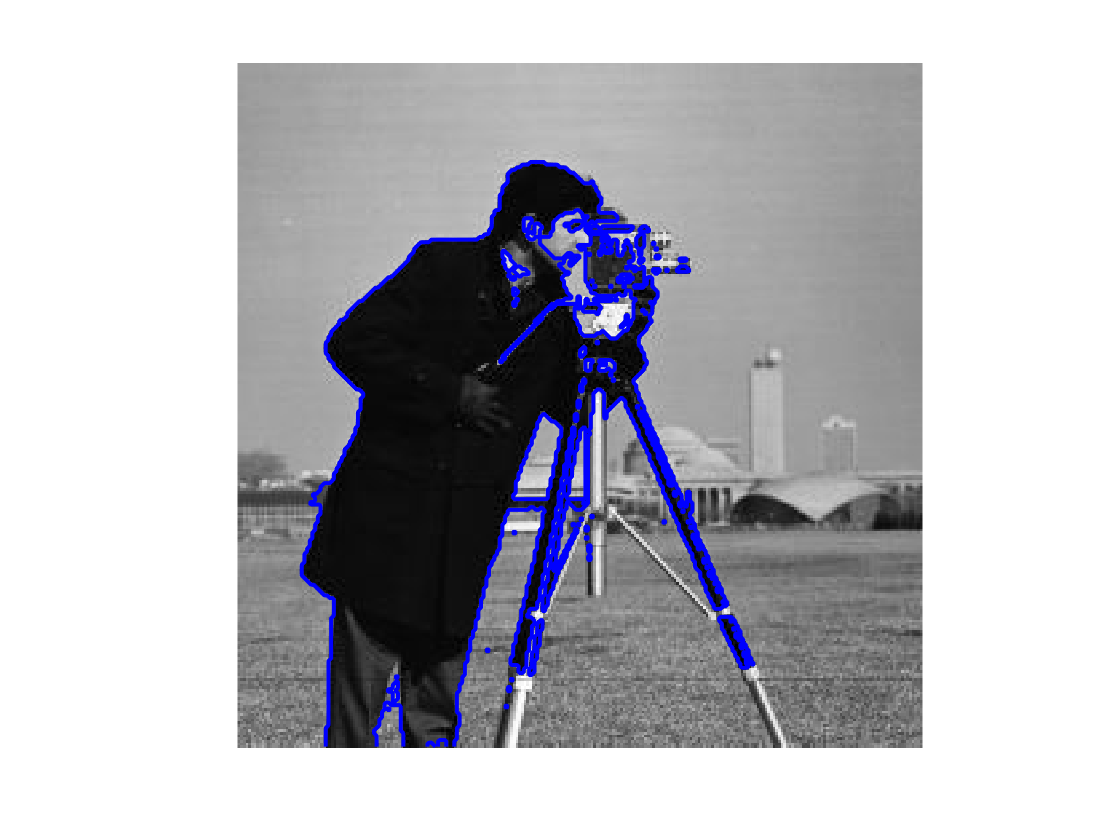}
\includegraphics[width = 0.3\textwidth]{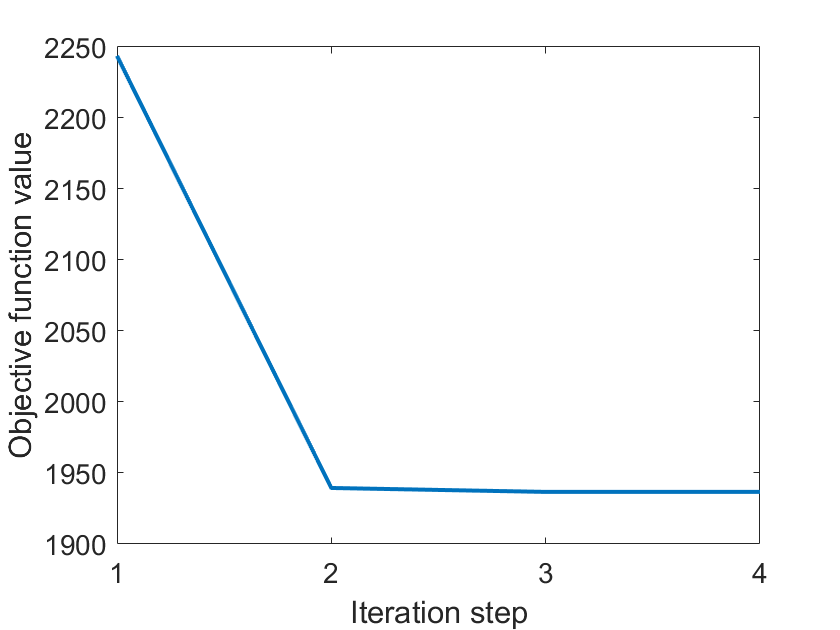}\\ \ 
\includegraphics[width = 0.3\textwidth]{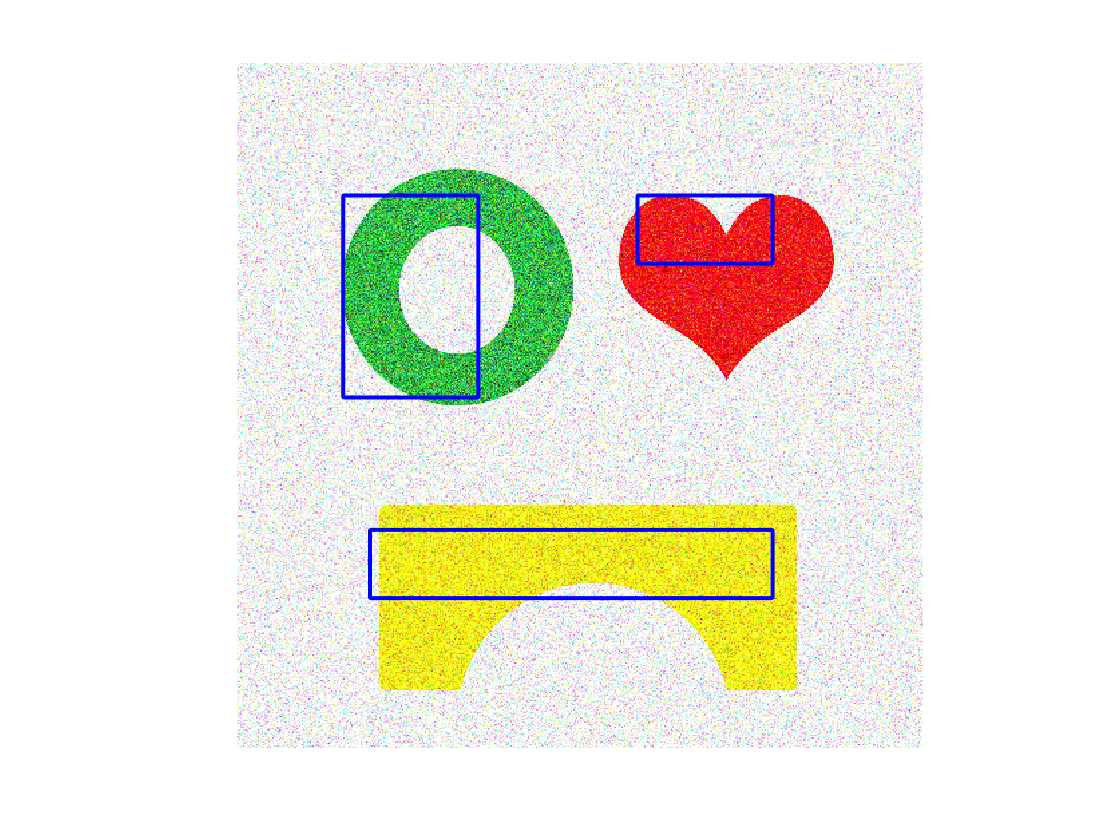}
\includegraphics[width = 0.3\textwidth]{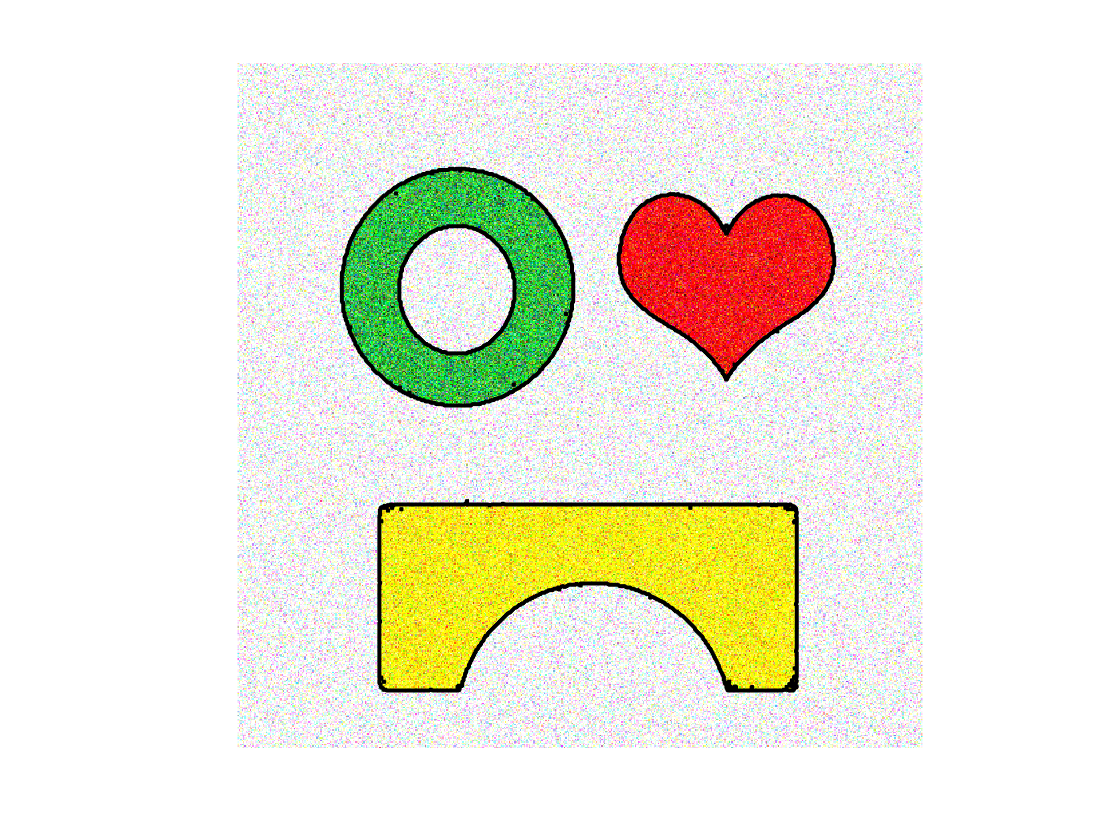}
\includegraphics[width = 0.3\textwidth]{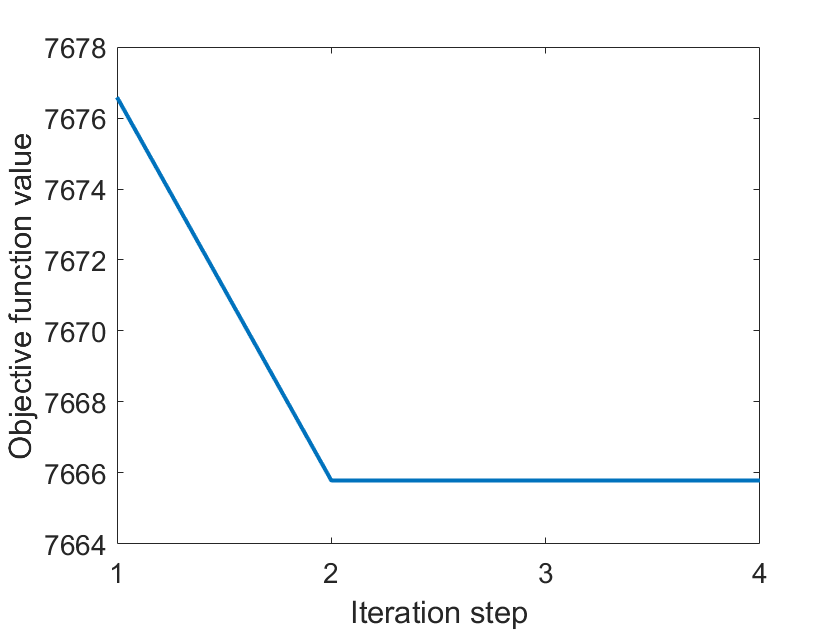}
\caption{Left: the initial guesses for the iteration. Middle: the converged results. Right: curves for the objective function values with respect to iteration step. In two images, $\lambda$ are set to be $0.03$ and $0.1$ respectively.  See Section~\ref{sec:cv}.}\label{fig1}
\end{figure}

The initial guesses are the characteristic functions of the squares bounded by blue lines. For example, in the gray image, the initial guess $u^0 = (u_1^0;u_2^0)$ is given as below,
\[u_{1,j}^0 = \begin{cases}
1 &\ \  \textrm{if $j$-th pixel is inside the region bounded by the blue curve}, \\
0 & \ \ \textrm{Otherwise},
\end{cases}
\]
and $u_{2,j} = 1-u_{1,j}$.

One can observe that the algorithm converges to an optimal solution in just 4 steps, implying the efficiency of the algorithm.

\subsection{Application to the Local Intensity Fitting (LIF) model for image segmentation}\label{sec:lif}

In this section,  we apply Algorithm~\ref{alg3} to the LIF model \cite{Chunming_Li_2008} for the two-phase case with the objective functional
\begin{align}
&\mathcal{E}_{LIF}(\Omega_1,\Omega_2, \theta_1,\theta_2)=\lambda  \sum_{i=1}^2 |\partial \Omega_i|  + \mu \sum_{i=1}^2 \int_{\Omega}\int_{\Omega_i}G_{\sigma}(x-y)|\theta_i(x)-I(y)|^2 \ dydx  \label{LBF}. 
\end{align}
where $\mu_i$ are given constants. 

Using the characteristic representation, approximation similar to that in \eqref{approxCV}, and discretization on pixels, one arrives at the following problem
\begin{align}
\min_{u \in C, \theta_i \in \mathbb R^p}  \Psi(u,\theta) =   \sum_{i=1}^2 \lambda (1-u_i) K_\tau u_{i}^T+ \mu \left(u_{i} K_\sigma {(\theta_{i}^2+I^2)}^T  - 2(u_{i}\circ I)K_\sigma \theta_{i}^T\right) \label{minLBF}
\end{align}
where $\theta_i\in \mathbb R^p$ is a vector for $\theta_i(x)$ on discrete pixels,   $K_\tau$ and $K_\sigma$ are two $p\times p$ matrices whose $(i,j)$-th entries are $G_\tau(x_i-x_j)$ and $G_\sigma(x_i-x_j)$, $\theta_i^2$ and $I^2$ are square of values at each entry, and $\circ$ denotes the Hadamard product.

In this case, using the fact that $K_\sigma$ is a symmetric positive definite matrix and  $$(u_{i}\circ I)K_\sigma \theta_{i}^T = \theta_{i} K_\sigma(u_{i}\circ I)^T \ \ {\rm and}   \ \ u_{i} K_\sigma {(\theta_{i}^2)}^T  = \theta_{i}^2 K_\sigma u_{i}^T .$$
subproblem \eqref{subproblem1} can also be explicitly solved by 
\[\theta_{i,j} = \frac{(K_\sigma(u_{i}\circ  I)^T)_j}{(K_\sigma u_{i}^T)_j}.\]
Here,  the subscript $j$ denotes the $j$-th entry in matrix vector multiplication.
To make the calculation of $\theta_{i,j}$ to be stable, one may use a regularized formula
\[\theta_{i,j}^{k+1} = \frac{(K_\sigma(u_{i}\circ  I)^T)_j+\varepsilon}{(K_\sigma u_{i}^T)_j+\varepsilon}.\]

One can apply Algorithm~\ref{alg3} to problem \eqref{minLBF} and the subproblem \eqref{subproblem2} for $u_{i,j}$ for each $j = 1, 2, \cdots, p$ at each iteration can be explicitly solved by iterating the follows;
\begin{align}
u_{i,j}^{k+1} = \begin{cases}
1 & \textrm{if} \ \ i = \min\left\{\arg\min_m (\psi_m^k)_j \right\}, \\
0 & \textrm{otherwise}
\end{cases}
\end{align}
where 
\[(\psi_m^k)_j =  \mu \left[K_\sigma {\left( {(\theta_{m}^{k+1})}^2+I^2\right)}^T  - 2 I^T \circ K_\sigma {(\theta_{m}^{k+1})}^T\right]_j+\lambda \left(K_\tau{(1-2u_m^k)}^T\right)_{j}.\]

We apply the algorithm into several intensity inhomogeneous images as listed in Figure~\ref{fig2}. In the first row, we use blue curve to implicitly denote the initial guesses. The segment results are shown in the second row with red curves. In both rows, we set $u_{1,j}=1$ at the pixel $j$ if it is located inside the curves and $u_{2,j}=1$ at the pixel $j$ if it is located outside the curves. The third row lists the corresponding objective function value decaying curves for the iterate starting from the initial guess. Again, one can observe that the algorithm can find the solution in very few steps.
In five images, the parameters $(\tau,\lambda,\mu,\sigma)$ are set to be $(5,1,150,3)$, $(3,1,245,3)$, $(10,1,110,3)$, $(2,1,90,3)$, $(3,1,80,3)$, respectively.

In the table of Figure~\ref{fig2}, we compare Algorithm~\ref{alg3} and the level-set method used by Li et al. \cite{Chunming_Li_2008} in terms of the number of iterations for convergence and the CPU time. One can easily observe that Algorithm~\ref{alg3} converges in significantly fewer iterations and a shorter CPU time, demonstrating its higher efficiency. The results obtained by   Li et al. \cite{Chunming_Li_2008} are similar to those in Figure~\ref{fig2}. More details and comparison can be referred to \cite{wang2019iterative}.

\begin{figure}[ht!]
\centering
\includegraphics[width = 0.19\textwidth]{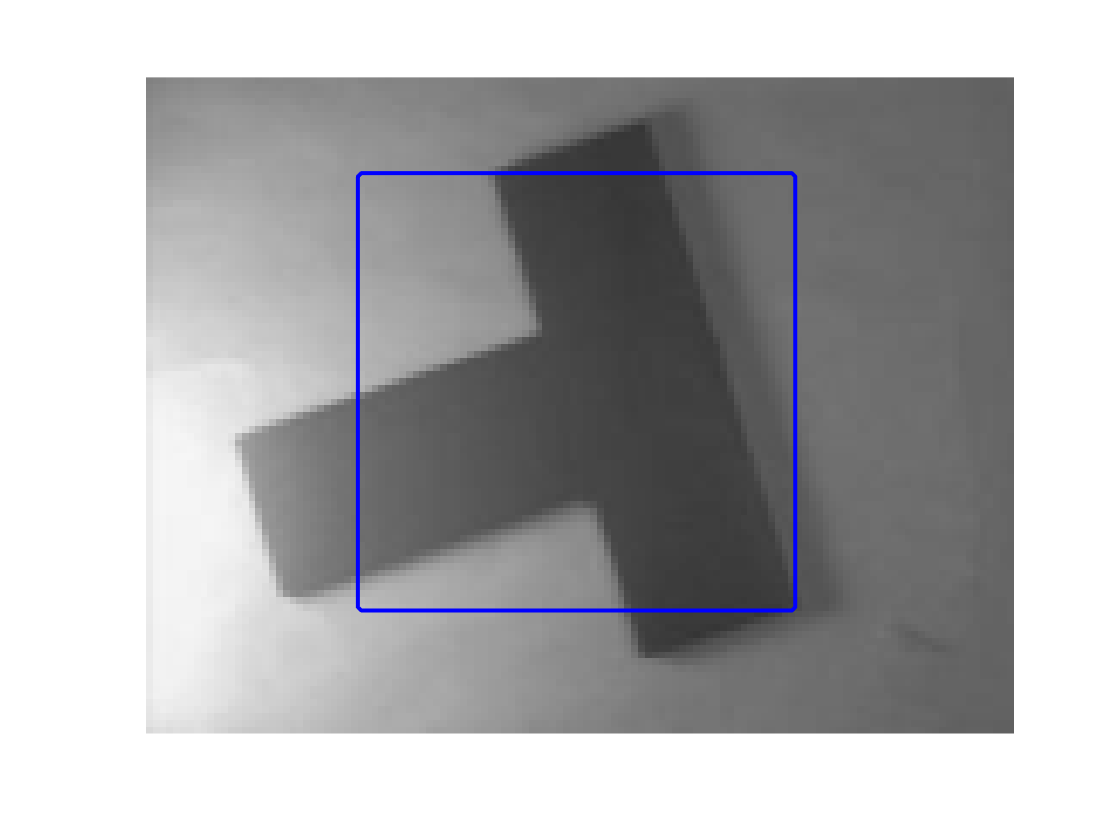}
\includegraphics[width = 0.19\textwidth]{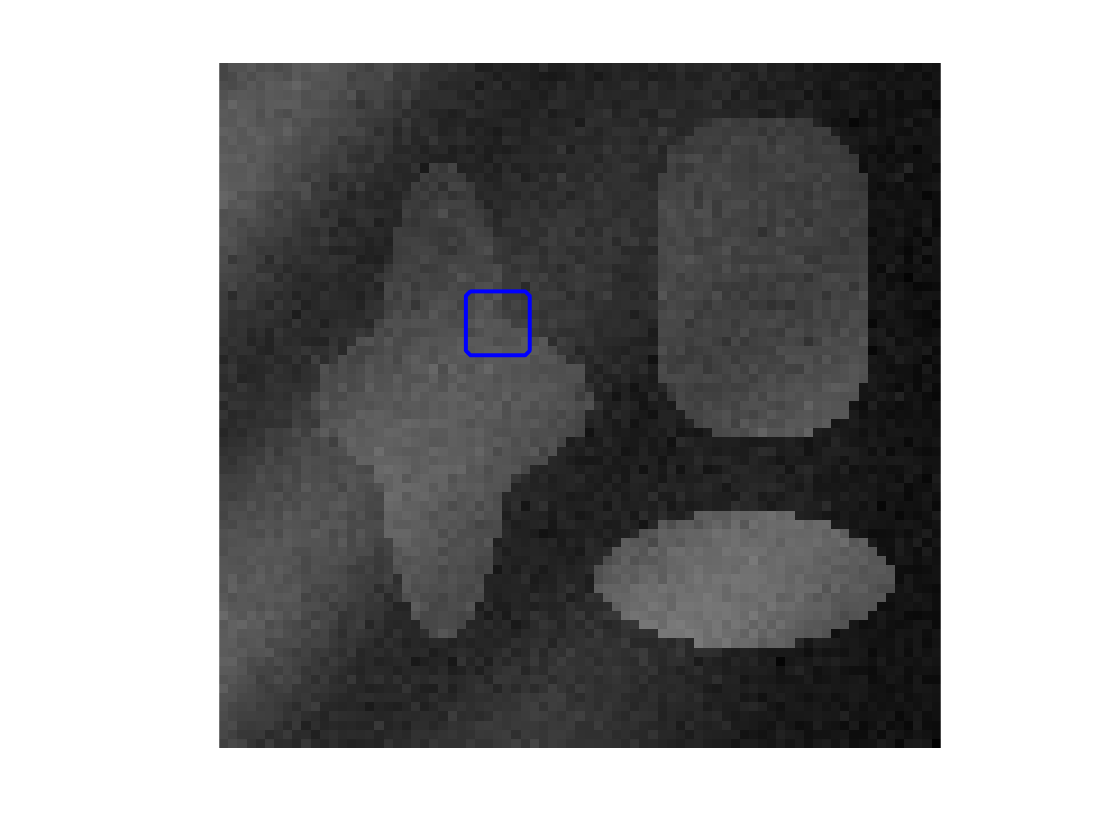}
\includegraphics[width = 0.19\textwidth]{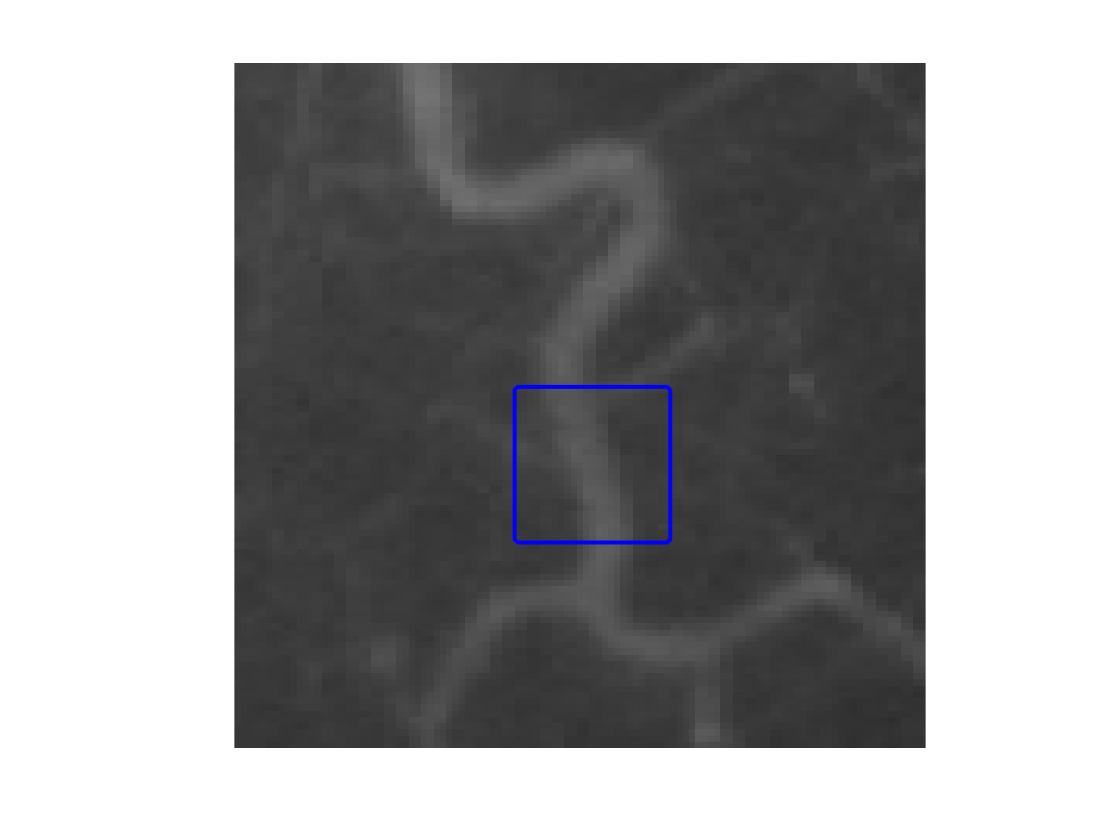}
\includegraphics[width = 0.19\textwidth]{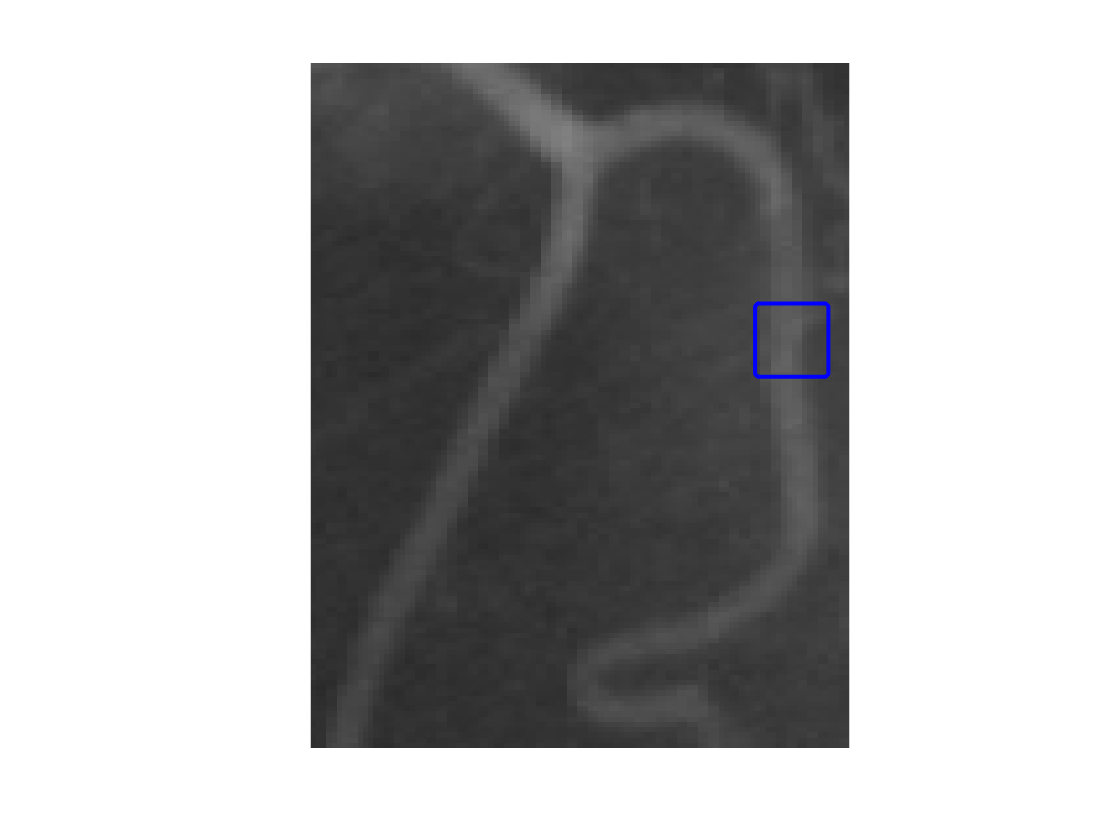}
\includegraphics[width = 0.19\textwidth]{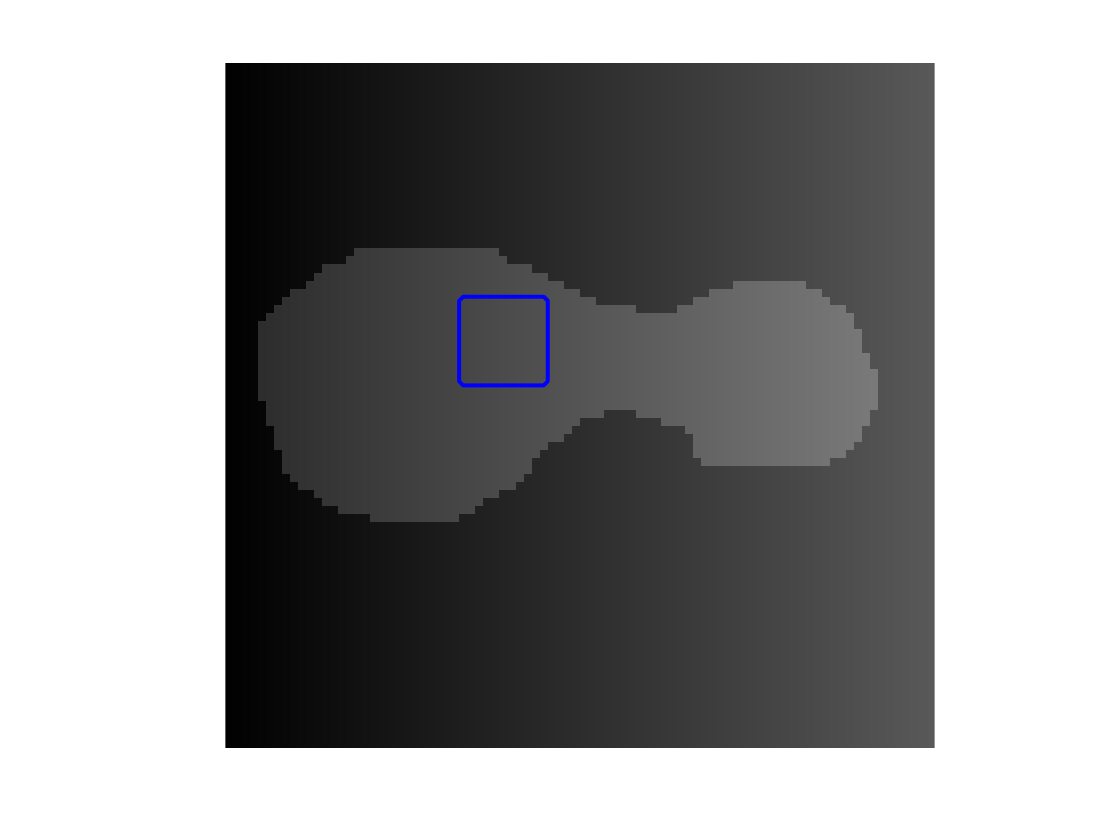}\\
\includegraphics[width = 0.19\textwidth]{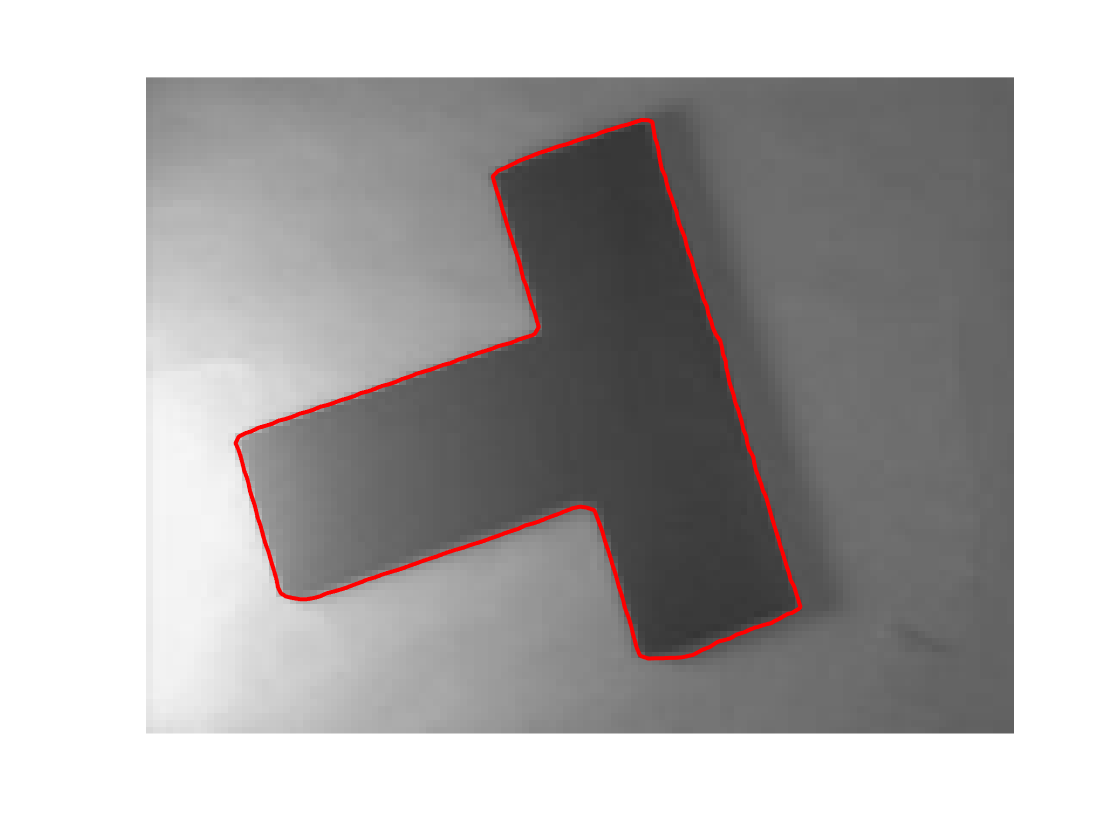}
\includegraphics[width = 0.19\textwidth]{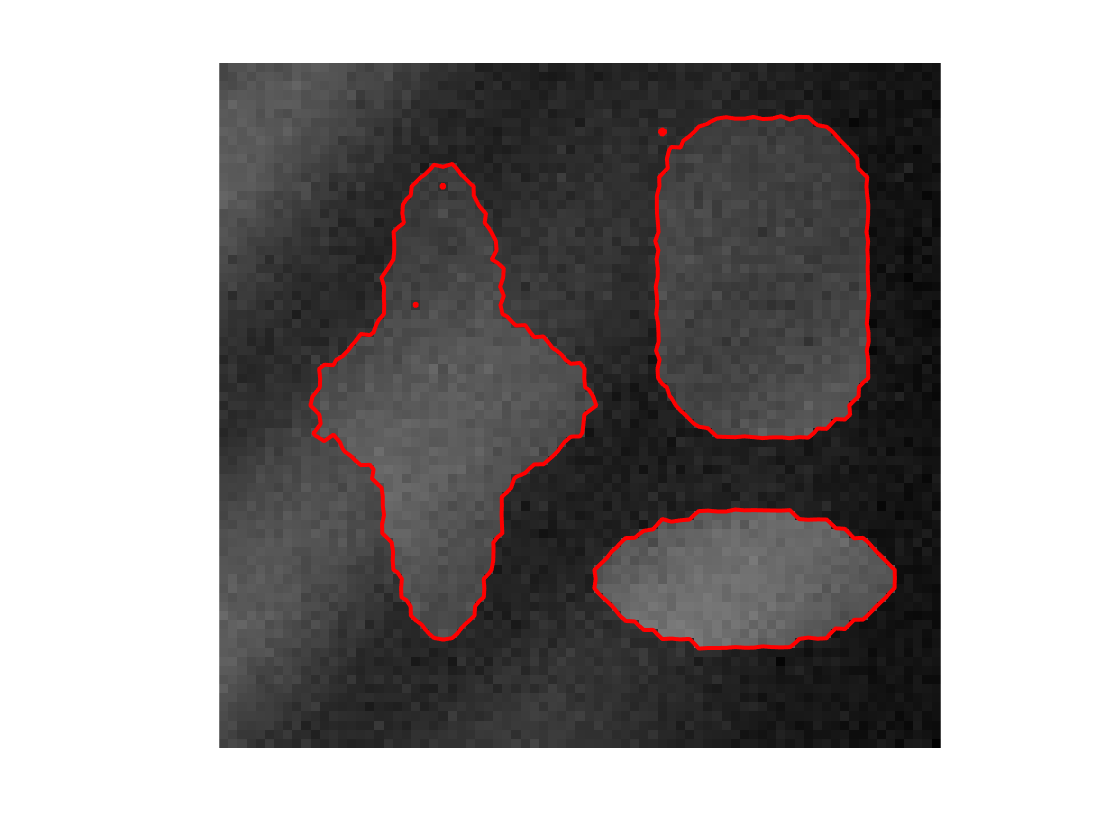}
\includegraphics[width = 0.19\textwidth]{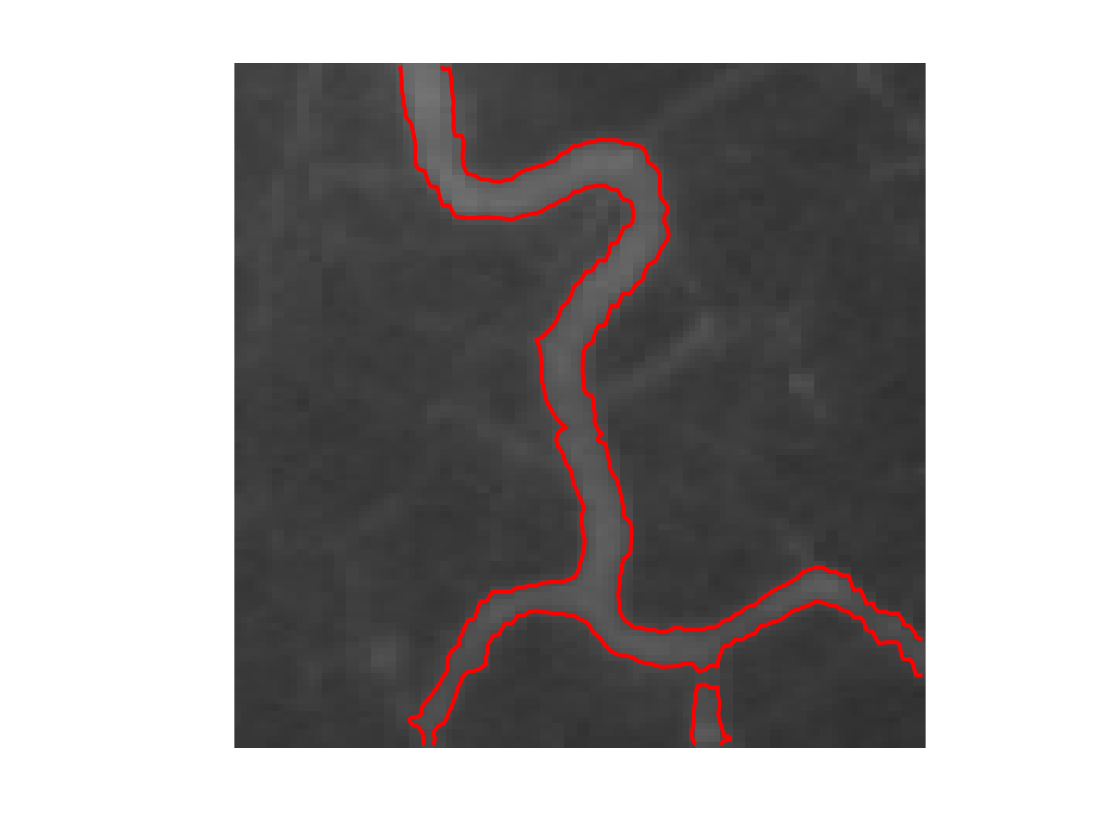}
\includegraphics[width = 0.19\textwidth]{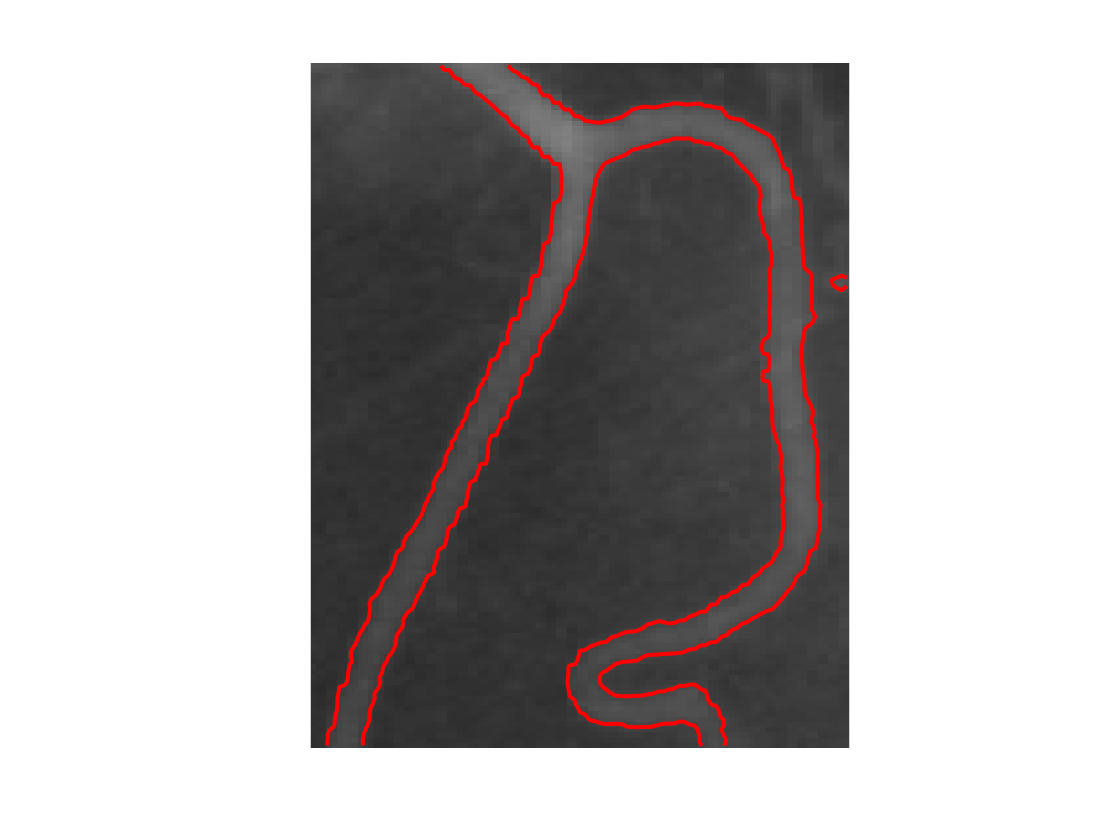}
\includegraphics[width = 0.19\textwidth]{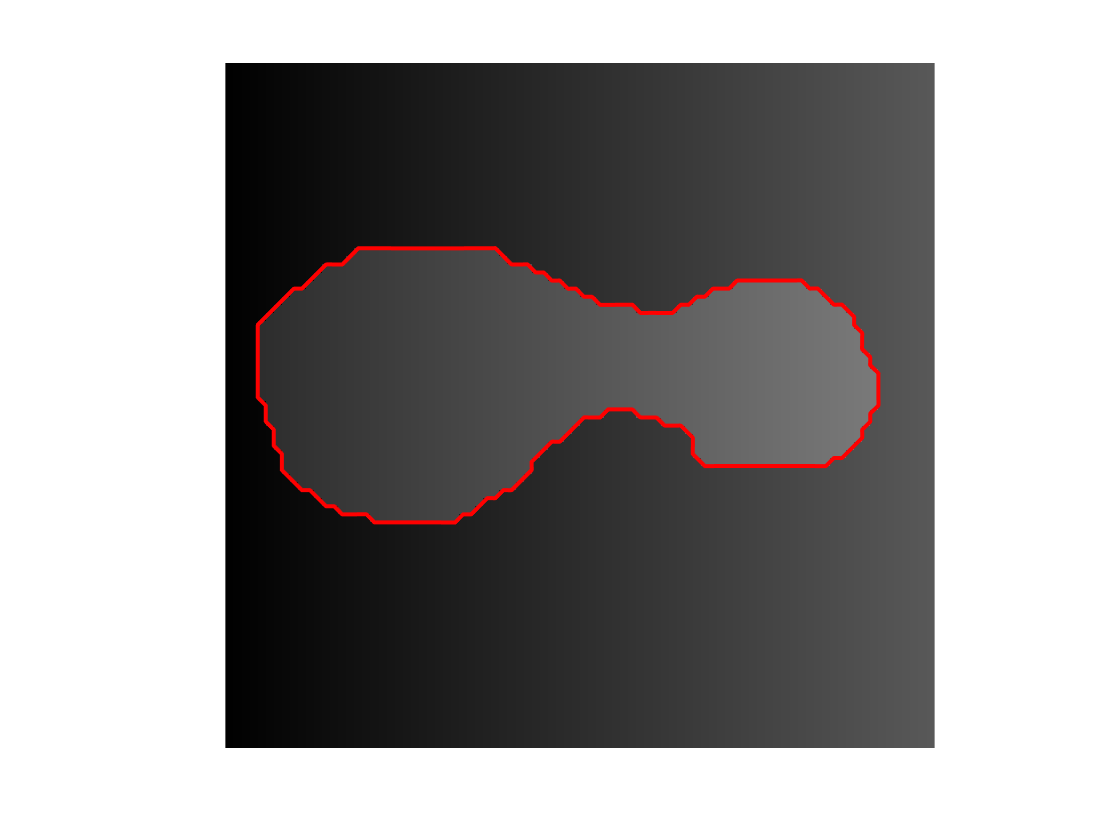}\\
\includegraphics[width = 0.19\textwidth]{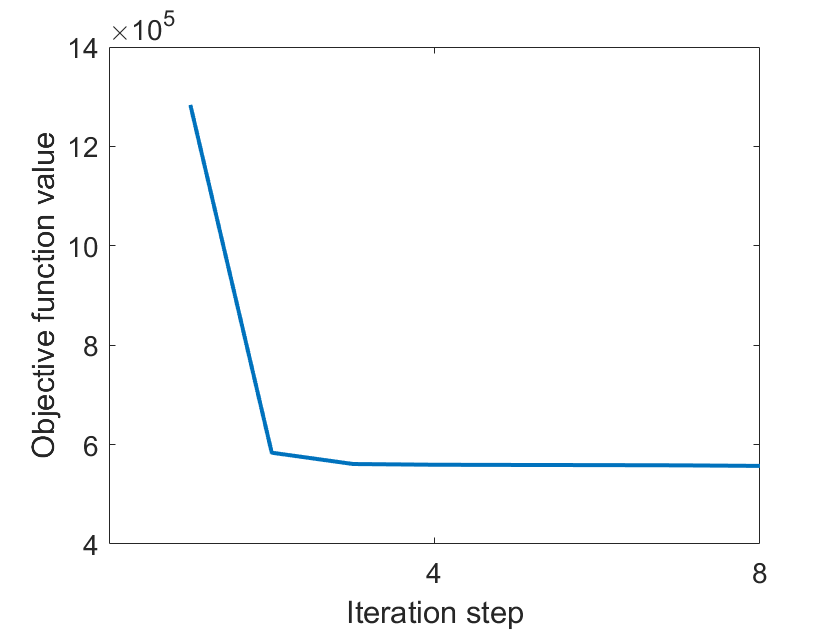}
\includegraphics[width = 0.19\textwidth]{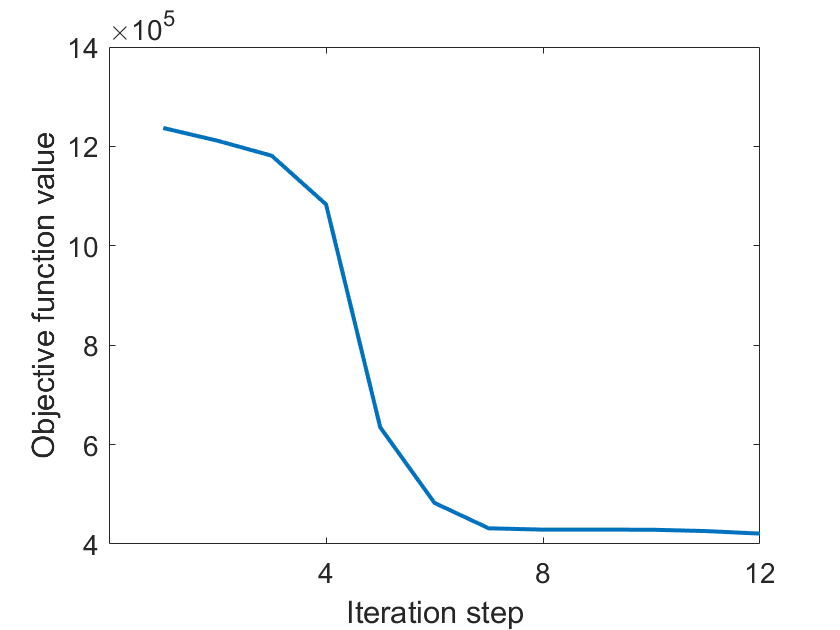}
\includegraphics[width = 0.19\textwidth]{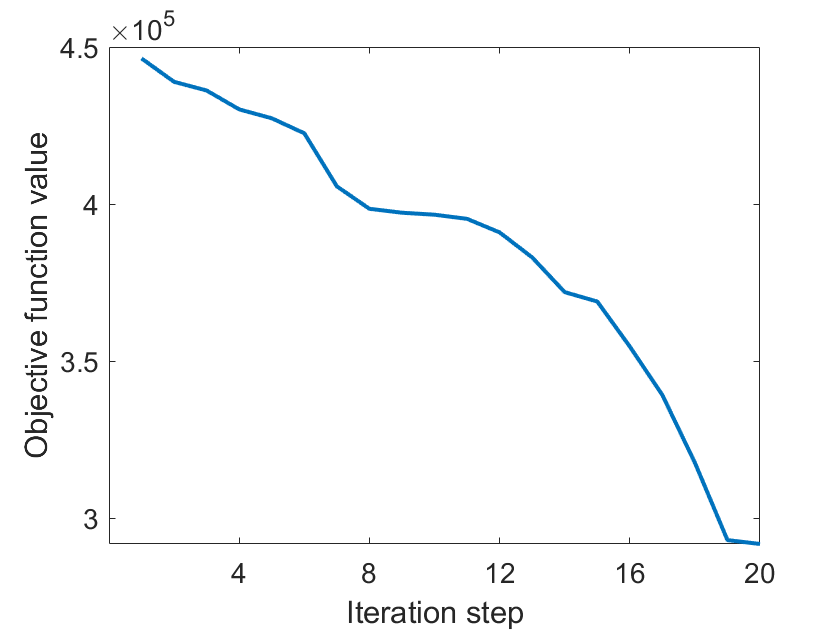}
\includegraphics[width = 0.19\textwidth]{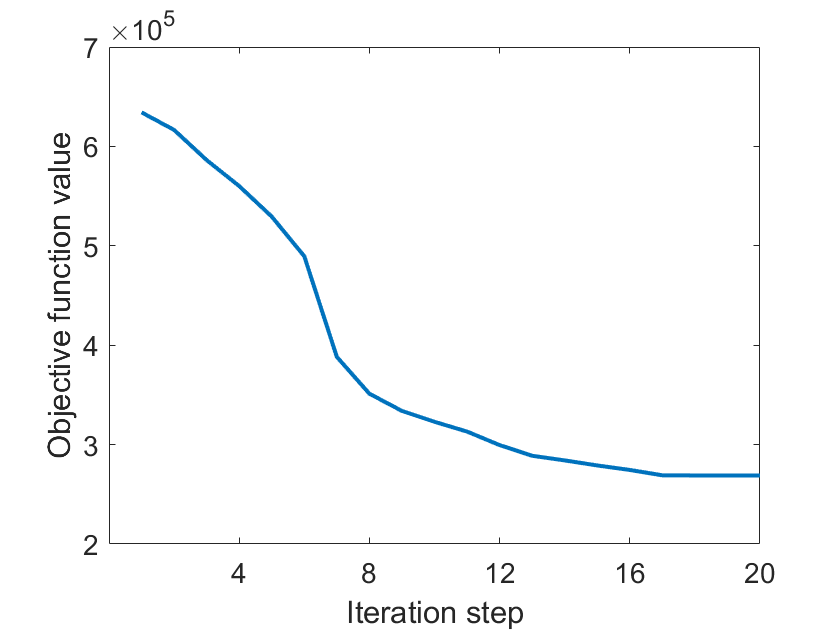}
\includegraphics[width = 0.19\textwidth]{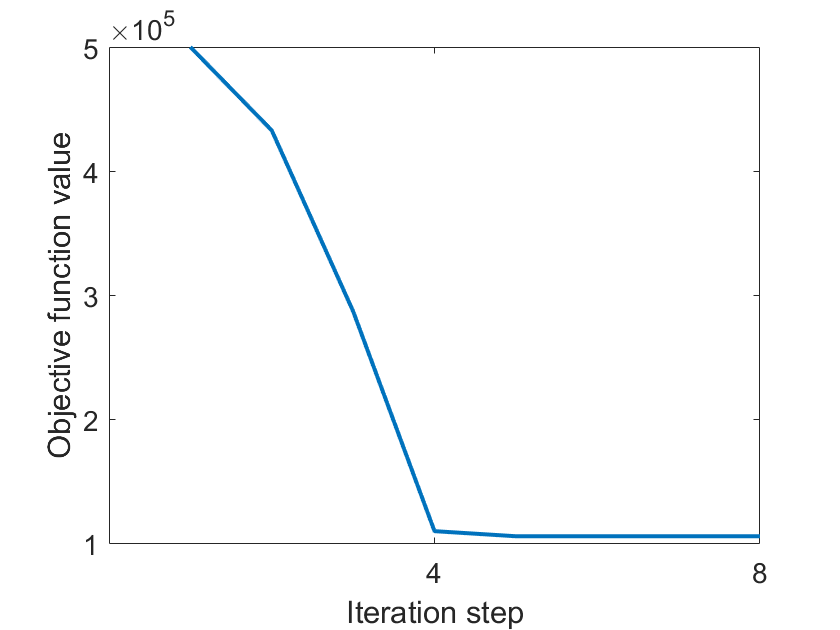}\\

\ \\
{\begin{tabular}{c|c|c|c|c|c|c}
\hline
\# of & Algorithm~\ref{alg3} & 6 & 12& 19 & 17 & 5  \\ 
\cline{2-7}
 iterations & Method in \cite{Chunming_Li_2008} & 256 & 131 & 117 & 209& 29\\
\hline
\hline
CPU time & Algorithm~\ref{alg3} & 0.08 & 0.07& 0.15 & 0.14& 0.06 \\
\cline{2-7}
(second)& Method in \cite{Chunming_Li_2008} &0.59  &0.43  & 0.48 &1.03& 0.15\\
\hline
\end{tabular}

\caption{First row: the initial guesses for the iteration. Second row: the converged results. Third row: curves for the objective function values with respect to iteration step. In five images, the parameters $(\tau,\lambda,\mu,\sigma)$ are set to be $(5,1,150,3)$, $(3,1,245,3)$, $(10,1,110,3)$, $(2,1,90,3)$, $(3,1,80,3)$. Table: comparison of the number of iterations and CPU time between Algorithm~\ref{alg3} and the level-set method. See Section~\ref{sec:lif}.}\label{fig2}
}
\end{figure}

\section{Conclusion and discussions} \label{sec:con}

In this paper, we gave a rigorous study on the convergence of the ICTM for solving a general problem raised from the discretization of many interface related optimization problems, including image segmentation, topology optimization, surface reconstruction, and so on. We studied three different situations where different solutions of minimization can be obtained. To our knowledge, this is the first systematical study on the convergence of the ICTM, which verifies the empirical observation from numerical experiments.

Similar analysis could be extended to the iterative method for Dirichlet partition problems \cite{wang2022efficient}, where a concave objective functional is minimized on a $L^2$ circle. The algorithmic framework can also be applied into data classification problems when one considers corresponding operators on point clouds. In addition, because of the small number of iterations for the minimization, it is straightforward to combine the algorithm with deep neural networks via the deep unfolding/unrolling ideas to have a mechanism and data driven methods for problems in image processing, surface reconstruction and topology optimization.

\section*{Acknowledgments}
 Dr. Wang acknowledges support from National Natural Science Foundation of China grant (Grant No. 12101524), Guangdong Basic and Applied Basic Research Foundation (Grant No. 2023A1515012199) and Shenzhen Science and Technology Innovation Program (Grant No. JCYJ20220530143803007, RCYX20221008092843046). Dr. Zhang  acknowledges support from NSFC 12222106, Shenzhen Science and
Technology Program (No. RCYX20200714114700072) and Guangdong Basic and Applied Basic Research Foundation 2022B1515020082.

\bibliographystyle{plain}
\bibliography{refs}

\begin{thebibliography}{10}

\bibitem{bae2011global}
Egil Bae, Jing Yuan, and Xue-Cheng Tai.
\newblock Global minimization for continuous multiphase partitioning problems
  using a dual approach.
\newblock {\em International journal of computer vision}, 92(1):112--129, 2011.

\bibitem{barles1995simple}
Guy Barles and Christine Georgelin.
\newblock A simple proof of convergence for an approximation scheme for
  computing motions by mean curvature.
\newblock {\em SIAM Journal on Numerical Analysis}, 32(2):484--500, 1995.

\bibitem{Borrvall_2002}
Thomas Borrvall and Joakim Petersson.
\newblock Topology optimization of fluids in stokes flow.
\newblock {\em International Journal for Numerical Methods in Fluids},
  41(1):77--107, 2002.

\bibitem{borwein2006convex}
Jonathan Borwein and Adrian Lewis.
\newblock Convex analysis and nonlinear optimization.
\newblock {\em CMS Books in Mathematics}, 2006.

\bibitem{boykov2004experimental}
Yuri Boykov and Vladimir Kolmogorov.
\newblock An experimental comparison of min-cut/max-flow algorithms for energy
  minimization in vision.
\newblock {\em IEEE transactions on pattern analysis and machine intelligence},
  26(9):1124--1137, 2004.

\bibitem{cai2013two}
Xiaohao Cai, Raymond Chan, and Tieyong Zeng.
\newblock A two-stage image segmentation method using a convex variant of the
  {Mumford--Shah} model and thresholding.
\newblock {\em SIAM Journal on Imaging Sciences}, 6(1):368--390, 2013.

\bibitem{chan2014two}
Raymond Chan, Hongfei Yang, and Tieyong Zeng.
\newblock A two-stage image segmentation method for blurry images with poisson
  or multiplicative gamma noise.
\newblock {\em SIAM Journal on Imaging Sciences}, 7(1):98--127, 2014.

\bibitem{Chan_2001}
T.F. Chan and L.A. Vese.
\newblock Active contours without edges.
\newblock {\em {IEEE} Transactions on Image Processing}, 10(2):266--277, 2001.

\bibitem{chen2018efficient}
Huangxin Chen, Haitao Leng, Dong Wang, and Xiao-Ping Wang.
\newblock An efficient threshold dynamics method for topology optimization for
  fluids.
\newblock {\em CSIAM Transactions on Applied Mathematics}, 3(1):26--56, 2022.

\bibitem{Chen_2002}
Long-Qing Chen.
\newblock Phase-field models for microstructure evolution.
\newblock {\em Annual Review of Materials Research}, 32(1):113--140, aug 2002.

\bibitem{Cherkaev_2011}
Andrej Cherkaev and Yuan Zhang.
\newblock Optimal anisotropic three-phase conducting composites: Plane problem.
\newblock {\em International Journal of Solids and Structures},
  48(20):2800--2813, oct 2011.

\bibitem{du2006convergence}
Qiang Du, Maria Emelianenko, and Lili Ju.
\newblock Convergence of the lloyd algorithm for computing centroidal {Voronoi}
  tessellations.
\newblock {\em SIAM Journal on Numerical Analysis}, 44(1):102--119, 2006.

\bibitem{du1999centroidal}
Qiang Du, Vance Faber, and Max Gunzburger.
\newblock {Centroidal Voronoi tessellations}: Applications and algorithms.
\newblock {\em SIAM review}, 41(4):637--676, 1999.

\bibitem{du2019maximum}
Qiang Du, Lili Ju, Xiao Li, and Zhonghua Qiao.
\newblock Maximum principle preserving exponential time differencing schemes
  for the nonlocal {Allen--Cahn} equation.
\newblock {\em SIAM Journal on Numerical Analysis}, 57(2):875--898, 2019.

\bibitem{Elsey_2017}
Matt Elsey and Selim Esedoglu.
\newblock Threshold dynamics for anisotropic surface energies.
\newblock {\em Mathematics of Computation}, 87(312):1721--1756, 2017.

\bibitem{esedoglu2017convolution}
Selim Esedoglu and Matt Jacobs.
\newblock Convolution kernels and stability of threshold dynamics methods.
\newblock {\em SIAM Journal on Numerical Analysis}, 55(5):2123--2150, 2017.

\bibitem{esedoglu2015threshold}
Selim Esedoglu and Felix Otto.
\newblock Threshold dynamics for networks with arbitrary surface tensions.
\newblock {\em Communications on Pure and Applied Mathematics}, 68(5):808--864,
  2015.

\bibitem{esedog2006threshold}
Selim Esedoglu, Yen-Hsi~Richard Tsai, et~al.
\newblock Threshold dynamics for the piecewise constant mumford--shah
  functional.
\newblock {\em Journal of Computational Physics}, 211(1):367--384, 2006.

\bibitem{glimm2003conservative}
James Glimm, Xiaolin Li, Yingjie Liu, Zhiliang Xu, and Ning Zhao.
\newblock Conservative front tracking with improved accuracy.
\newblock {\em SIAM Journal on Numerical Analysis}, 41(5):1926--1947, 2003.

\bibitem{he2019fast}
Yuchen He, Martin Huska, Sung~Ha Kang, and Hao Liu.
\newblock Fast algorithms for surface reconstruction from point cloud.
\newblock {\em arXiv preprint arXiv:1907.01142}, 2019.

\bibitem{Hu_2013}
Dan Hu and David Cai.
\newblock Adaptation and optimization of biological transport networks.
\newblock {\em Physical Review Letters}, 111(13), sep 2013.

\bibitem{hu2022}
Wei Hu, Dong Wang, and Xiao-Ping Wang.
\newblock An efficient iterative method for the formulation of flow networks.
\newblock {\em Communications in Computational Physics}, 31(5):1317--1340,
  2022.

\bibitem{jiang2022topology}
Fuhang Jiang, Leilei Chen, Jie Wang, Xiaofei Miao, and Haibo Chen.
\newblock Topology optimization of multimaterial distribution based on
  isogeometric boundary element and piecewise constant level set method.
\newblock {\em Computer Methods in Applied Mechanics and Engineering},
  390:114484, 2022.

\bibitem{laux2016convergence2}
Tim Laux and Felix Otto.
\newblock Convergence of the thresholding scheme for multi-phase mean-curvature
  flow.
\newblock {\em Calculus of Variations and Partial Differential Equations},
  55(5):129, 2016.

\bibitem{li2020convergence}
Buyang Li.
\newblock Convergence of dziuk's linearly implicit parametric finite element
  method for curve shortening flow.
\newblock {\em SIAM Journal on Numerical Analysis}, 58(4):2315--2333, 2020.

\bibitem{Chunming_Li_2008}
Chunming Li, Chiu-Yen Kao, J.C. Gore, and Zhaohua Ding.
\newblock Minimization of region-scalable fitting energy for image
  segmentation.
\newblock {\em {IEEE} Transactions on Image Processing}, 17(10):1940--1949, oct
  2008.

\bibitem{Li_2007}
Chunming Li, Chiu-Yen Kao, John~C. Gore, and Zhaohua Ding.
\newblock Implicit active contours driven by local binary fitting energy.
\newblock In {\em 2007 {IEEE} Conference on Computer Vision and Pattern
  Recognition}. {IEEE}, jun 2007.

\bibitem{qiao2021two}
Chaoyu Liu, Zhonghua Qiao, and Qian Zhang.
\newblock Two-phase segmentation for intensity inhomogeneous images by the
  {Allen--Cahn} local binary fitting model.
\newblock {\em SIAM Journal on Scientific Computing}, 44(1):B177--B196, 2022.

\bibitem{liu2011fast}
Jun Liu, Xue-cheng Tai, Haiyang Huang, and Zhongdan Huan.
\newblock A fast segmentation method based on constraint optimization and its
  applications: Intensity inhomogeneity and texture segmentation.
\newblock {\em Pattern Recognition}, 44(9):2093--2108, 2011.

\bibitem{lu2021efficient}
Song Lu and Xianmin Xu.
\newblock An efficient diffusion generated motion method for wetting dynamics.
\newblock {\em Journal of Computational Physics}, 441:110476, 2021.

\bibitem{luo2023binary}
Shousheng Luo, Jinfeng Chen, Yunhai Xiao, and Xue-Cheng Tai.
\newblock A binary characterization method for shape convexity and
  applications.
\newblock {\em Applied Mathematical Modelling}, 122:780--795, 2023.

\bibitem{luo2020convex}
Shousheng Luo, Xue-Cheng Tai, and Yang Wang.
\newblock Convex shape representation with binary labels for image
  segmentation: Models and fast algorithms.
\newblock {\em arXiv preprint arXiv:2002.09600}, 2020.

\bibitem{Ma_2021}
Jun Ma, Dong Wang, Xiao-Ping Wang, and Xiaoping Yang.
\newblock A characteristic function-based algorithm for geodesic active
  contours.
\newblock {\em {SIAM} Journal on Imaging Sciences}, 14(3):1184--1205, Jan 2021.

\bibitem{merkurjev2013mbo}
Ekaterina Merkurjev, Tijana Kostic, and Andrea~L Bertozzi.
\newblock An {MBO} scheme on graphs for classification and image processing.
\newblock {\em {SIAM} Journal on Imaging Sciences}, 6(4):1903--1930, 2013.

\bibitem{merriman1992diffusion}
B.~Merriman, J.~K. Bence, and S.~Osher.
\newblock Diffusion generated motion by mean curvature.
\newblock UCLA CAM Report 92-18, 1992.

\bibitem{merriman2000convolution}
Barry Merriman and Steven~J Ruuth.
\newblock Convolution-generated motion and generalized {H}uygens' principles
  for interface motion.
\newblock {\em SIAM Journal on Applied Mathematics}, 60(3):868--890, 2000.

\bibitem{Osher_1988}
Stanley Osher and James~A Sethian.
\newblock Fronts propagating with curvature-dependent speed: Algorithms based
  on hamilton-jacobi formulations.
\newblock {\em Journal of Computational Physics}, 79(1):12--49, nov 1988.

\bibitem{shen2018convergence}
Jie Shen and Jie Xu.
\newblock Convergence and error analysis for the scalar auxiliary variable
  (sav) schemes to gradient flows.
\newblock {\em SIAM Journal on Numerical Analysis}, 56(5):2895--2912, 2018.

\bibitem{Sigmund_2013}
Ole Sigmund and Kurt Maute.
\newblock Topology optimization approaches.
\newblock {\em Structural and Multidisciplinary Optimization},
  48(6):1031--1055, aug 2013.

\bibitem{tai2007image}
Xue-Cheng Tai, Oddvar Christiansen, Ping Lin, and Inge Skj{\ae}laaen.
\newblock Image segmentation using some piecewise constant level set methods
  with {MBO} type of projection.
\newblock {\em International Journal of Computer Vision}, 73:61--76, 2007.

\bibitem{tai2023potts}
Xuecheng Tai, Lingfeng Li, and Egil Bae.
\newblock The {Potts} model with different piecewise constant representations
  and fast algorithms: a survey.
\newblock In {\em Handbook of Mathematical Models and Algorithms in Computer
  Vision and Imaging: Mathematical Imaging and Vision}, pages 1--41. Springer,
  2023.

\bibitem{Unverdi_1992}
Salih~Ozen Unverdi and Gr{\'{e}}tar Tryggvason.
\newblock A front-tracking method for viscous, incompressible, multi-fluid
  flows.
\newblock {\em Journal of Computational Physics}, 100(1):25--37, may 1992.

\bibitem{wan2012reconstructing}
Min Wan, Yu~Wang, Egil Bae, Xue-Cheng Tai, and Desheng Wang.
\newblock Reconstructing open surfaces via graph-cuts.
\newblock {\em IEEE transactions on visualization and computer graphics},
  19(2):306--318, 2012.

\bibitem{Wang_2021}
Dong Wang.
\newblock An efficient iterative method for reconstructing surface from point
  clouds.
\newblock {\em Journal of Scientific Computing}, 87(1), mar 2021.

\bibitem{wang2022efficient}
Dong Wang.
\newblock An efficient unconditionally stable method for dirichlet partitions
  in arbitrary domains.
\newblock {\em to appear in SIAM Journal on Scientific Computing}, 2022.

\bibitem{wang2016efficient}
Dong Wang, Haohan Li, Xiaoyu Wei, and Xiao-Ping Wang.
\newblock An efficient iterative thresholding method for image segmentation.
\newblock {\em Journal of Computational Physics}, 350(1):657--667, 2017.

\bibitem{wang2019iterative}
Dong Wang and Xiao-Ping Wang.
\newblock The iterative convolution--thresholding method ({ICTM}) for image
  segmentation.
\newblock {\em Pattern Recognition}, page 108794, 2022.

\bibitem{WWX2018}
Dong Wang, Xiao-Ping Wang, and Xianmin Xu.
\newblock An improved threshold dynamics method for wetting dynamics.
\newblock {\em J. Comput. Phys.}, 392:291--310, 2019.

\bibitem{wei2009piecewise}
Peng Wei and Michael~Yu Wang.
\newblock Piecewise constant level set method for structural topology
  optimization.
\newblock {\em International Journal for Numerical Methods in Engineering},
  78(4):379--402, 2009.

\bibitem{xu2016efficient}
Xianmin Xu, Dong Wang, and Xiao-Ping Wang.
\newblock An efficient threshold dynamics method for wetting on rough surfaces.
\newblock {\em Journal of Computational Physics}, 330(1):510--528, 2017.

\bibitem{ying2021adaptive}
Xianmin Xu and Wenjun Ying.
\newblock An adaptive threshold dynamics method for three-dimensional wetting
  on rough surfaces.
\newblock {\em Communications in Computational Physics}, 29(1):57--79, 2021.

\bibitem{yuan2010continuous}
Jing Yuan, Egil Bae, Xue-Cheng Tai, and Yuri Boykov.
\newblock A continuous max-flow approach to potts model.
\newblock In {\em Computer Vision--ECCV 2010: 11th European Conference on
  Computer Vision, Heraklion, Crete, Greece, September 5-11, 2010, Proceedings,
  Part VI 11}, pages 379--392. Springer, 2010.

\bibitem{Zhang_2016}
Kaihua Zhang, Lei Zhang, Kin-Man Lam, and David Zhang.
\newblock A level set approach to image segmentation with intensity
  inhomogeneity.
\newblock {\em {IEEE} Transactions on Cybernetics}, 46(2):546--557, feb 2016.

\bibitem{zhang2018approach}
Zhengfang Zhang and Weifeng Chen.
\newblock An approach for maximizing the smallest eigenfrequency of structure
  vibration based on piecewise constant level set method.
\newblock {\em Journal of Computational Physics}, 361:377--390, 2018.

\bibitem{Hong_Kai_Zhao}
Hong-Kai Zhao, S.~Osher, and R.~Fedkiw.
\newblock Fast surface reconstruction using the level set method.
\newblock In {\em Proceedings {IEEE} Workshop on Variational and Level Set
  Methods in Computer Vision}. {IEEE} Computer Soc.

\bibitem{Zhao_2000}
Hong-Kai Zhao, Stanley Osher, Barry Merriman, and Myungjoo Kang.
\newblock Implicit and nonparametric shape reconstruction from unorganized data
  using a variational level set method.
\newblock {\em Computer Vision and Image Understanding}, 80(3):295--314, dec
  2000.

\end{thebibliography}


\end{document}